\documentclass[11pt, a4paper]{amsart}
\usepackage{amsmath}
\usepackage[OT4]{fontenc}
\usepackage{amssymb}
\usepackage{amsthm}
\usepackage{graphicx}
\usepackage{url}
\usepackage{enumerate}
\usepackage{xcolor}
\usepackage{mathrsfs} 
\usepackage{mathtools}
\usepackage{t1enc}
\usepackage{lineno}
\usepackage[margin=3cm]{geometry}
\usepackage{comment}

\newtheorem{thm}{Theorem}[section]
\newtheorem{lem}[thm]{Lemma}

\newtheorem{cor}[thm]{Corollary}

\newtheorem{fact}[thm]{Fact}

\theoremstyle{definition}
\newtheorem{defn}[thm]{Definition}
\newtheorem{rem}[thm]{Remark}
\newtheorem{ex}[thm]{Example}
\DeclareMathOperator{\supp}{supp} 

\DeclareMathOperator{\dist}{dist}

\newcommand{\f}{\bar{f}}

\newcommand{\di}{D_P}
\DeclareMathOperator{\m}{\mathfrak m}

\DeclareMathOperator{\Gen}{Gen}

\newcommand{\und}[1]{\underline{#1}}

\newcommand{\om}[1]{V({#1})}
\newcommand{\omT}[1]{V_T({#1})}

\DeclareMathOperator*{\Limtop}{Lim\,top}
\DeclareMathOperator*{\Lstop}{Ls\,top}
\DeclareMathOperator*{\Litop}{Li\,top}

\newcommand{\Lts}{\Lstop}
\newcommand{\Lti}{\Litop}

\renewcommand{\phi}{\varphi}

\DeclareMathOperator{\M}{\mathcal{M}}
\DeclareMathOperator{\MT}{\mathcal{M}_T}

\DeclareMathOperator{\MTe}{\M_T^e}

\usepackage[foot]{amsaddr}

\newcommand{\eps}{\varepsilon}
\newcommand{\R}{\mathbb{R}}
\newcommand{\Z}{\mathbb{Z}}

\newcommand{\N}{\mathbb{N}}

\newcommand{\Dom}{\mathcal{D}}
\newcommand{\Ran}{\mathcal{R}}
\newcommand{\fbar}{\bar{f}}
\newcommand{\Fbar}{\bar{\textit{fk}}}
\newcommand{\dbar}{\bar{d}}
\newcommand{\XX}{\mathbf{X}}
\newcommand{\bP}{\mathbf{P}}
\newcommand{\code}{\und{\mathcal{P}}}

\newcommand{\cP}{\mathcal{P}}
\newcommand{\cR}{\mathcal{R}}

\title[FK pseudometric and GIKN nonhyperbolic measures]{Feldman-Katok pseudometric and the GIKN construction of royal nonhyperbolic ergodic measures}

\author{Dominik Kwietniak \and Martha \L{}\k{a}cka}

\address[D. Kwietniak]{
Faculty of Mathematics and Computer Science, Jagiellonian University in Krak\'ow, ul. \L o\-jasiewicza 6, 30-348 Krak\'ow, Poland
ORCID ID: 0000-0002-7794-2835
}
\email{dominik.kwietniak@uj.edu.pl}
\urladdr{www.im.uj.edu.pl/DominikKwietniak/}
\address[M. \L{}\k{a}cka]{
Faculty of Mathematics and Computer Science, Jagiellonian University in Krak\'ow, ul. \L o\-jasiewicza 6, 30-348 Krak\'ow, Poland, ORCID ID: 0000-0002-7511-9095}
\email{martha.ubik@uj.edu.pl}
\urladdr{www2.im.uj.edu.pl/MarthaLacka/}

\begin{document}

\begin{abstract}

We introduce Feldman-Katok convergence for 
invariant measures of a topological dynamical system. It can be seen as a counterpart to the convergence with respect to \emph{f-bar} metric $\fbar$ for finite-state stationary processes (shift-invariant measures on a symbolic space).  Feldman-Katok convergence is based on a dynamically defined Feldman-Katok pseudometric. This convergence is stronger than weak$^*$ convergence. We prove that Feldman-Katok convergence preserves ergodicity and makes the Kolmogorov-Sinai entropy lower semicontinuous, thereby preserving zero entropy.

We apply our findings to nonhyperbolic (having at least one vanishing Lyapunov exponent) \emph{royal measures} (ergodic measures constructed using the GIKN method  --- a method introduced by Gorodetski, Ilyashenko, Kleptsyn, and Nalsky in [\emph{Functional Analysis and its Applications}, \textbf{39} (2005), 21--30]). 
This construction scheme has been widely adapted to produce nonhyperbolic ergodic measures for diffeomorphisms of compact manifolds. We prove that each ergodic measure obtained through the GIKN method is the Feldman-Katok limit of a sequence of periodic measures, which implies it is either a periodic measure or it is loosely Kronecker (Kakutani equivalent to an aperiodic ergodic rotation on a compact group) and has zero entropy. This classifies all these measures up to Kakutani equivalence and confirms that geometric constructions of nonhyperbolic measures via periodic approximations based on the GIKN method systematically produce zero-entropy systems.

\end{abstract}
\maketitle


\section{Introduction}
A topological dynamical system is a pair $(X,T)$, where $X$ is a compact metrisable space and $T\colon X\to X$ is a continuous map.  
The weak$^*$ topology is the usual choice for the topology on the space of invariant measures $\MT(X)$ of $(X,T)$. While this topology is compact and metrisable, it is incompatible with several dynamical properties. For many dynamical systems, the entropy function on $\MT(X)$ endowed with the  weak$^*$ topology fails to be continuous: low-entropy measures can approximate high-entropy measures arbitrarily closely, and zero-entropy measures are often dense even when non-zero entropy measures exist. Similarly,  mixing and ergodicity are not preserved under weak$^*$ convergence. For shift-invariant measures on a shift (symbolic) space $\mathscr A^\infty$ over a finite alphabet $\mathscr A$, Feldman's $\fbar$-metric provides a stronger topology that is better adapted to dynamical behaviour. An even stronger topology is given by Ornstein's $\dbar$-metric, which is complete but non-compact (even non-separable). Both metrics, $\dbar$ and $\fbar$ work only for finite-state stationary processes (shift-invariant measures on symbolic space $\mathscr A^\infty$). 

Here, we introduce \emph{Feldman-Katok convergence} for invariant measures of topological dynamical systems. It serves as a counterpart to convergence with respect to the f-bar metric $\fbar$ for not necessarily symbolic systems and measures. This convergence is based on the \emph{Feldman-Katok pseudometric} $\Fbar$ measuring the similarity between orbits of a topological dynamical system. It is inspired by the $\fbar$ pseudometric for infinite sequences of symbols, which is in turn based on edit distance. The edit distance between two strings $u$ and $v$ of length $n$ with entries coming from a finite set of symbols is  $\fbar_n(u,v)=k/n$, where $k$ is the minimum number of symbols that must be removed from each string to make the remaining parts identical. Thus, two strings (words) are close in the edit metric if you can make them match by removing only a small fraction of their symbols. For infinite sequences of symbols $x=x_0x_1x_2\ldots$ and $y=y_0y_1y_2\ldots$ we set
\[
\fbar(x,y)=\limsup_{n\to\infty}\fbar_n(x_0x_1\ldots x_{n-1},y_0y_1\ldots y_{n-1}).
\]
We extend this idea to orbits of arbitrary dynamical systems. The Feldman-Katok pseudometric $\Fbar$ captures a similar idea of approximate matching with allowed discrepancies. We may ignore some portions of the orbits while requiring the remaining points on each orbit to be matched into pairs close in the usual metric on the space. We discuss the details below. It is not clear whether $\Fbar$ corresponds to a metric on the space of $T$-invariant measures $\MT(X)$ as it happens for the pseudometric $\fbar$ on $\mathscr A^\infty$ given by \eqref{eq:fbar} that corresponds to a metric on the space of shift-invariant measures $\M_{\sigma }(\mathscr A^\infty)$ given by \eqref{eq:fbar-metric-for-measures}. Nevertheless, we may use $\Fbar$ to define a certain notion of ``convergence'' for $\MT(X)$.

Roughly speaking, we say that a sequence $(\mu_n)_{n=1}^\infty$ of $T$-invariant  measures converges in the Feldman-Katok sense if there is a sequence $(x_n)_{n=1}^\infty$ which is a Cauchy sequence with respect to $\Fbar$ pseudometric such that for each $n\in\N$ the point $x_n$ is generic for $\mu_n$. 

Our definition of $\Fbar$ and Feldman-Katok convergence draws inspiration from the way Ornstein's dbar metric $\dbar$ and the corresponding pseudometric $\dbar$ for infinite sequences of symbols were generalised to Besicovitch pseudometrics and a metric on the set of invariant measures  \cite{BCKO, KLO,Shields, Weiss}. 

An $\Fbar$-Cauchy sequence of generic points naturally determines a measure, which we call the \emph{$\Fbar$-limit} of the sequence. 
We also prove that Feldman-Katok convergence is stronger than the standard weak$^*$ convergence of measures, because the $\Fbar$-limit $\mu$ must also be a limit with respect to the weak$^*$ topology of the sequence of measures corresponding to the generic points forming $\Fbar$-Cauchy sequence (Theorem \ref{thm:weak}). 
The  $\Fbar$-limit of a sequence of ergodic points (that is, points generic for ergodic measures) is ergodic (Theorem \ref{thm:Cauchy}), thus ergodicity is closed with respect to the Feldman-Katok convergence.  We show that the Feldman-Katok convergence also preserves the loosely Kronecker property: if a sequence of invariant measures converges in the Feldman-Katok sense and each measure in the sequence is loosely Kronecker, then the measure which is an $\Fbar$-limit of a corresponding sequence of generic points is also loosely Kronecker (Theorem \ref{thm:looselyK}). Moreover, we establish that the Kolmogorov-Sinai entropy function is lower semicontinuous with respect to Feldman-Katok convergence (Theorem \ref{lsc}), which immediately implies that this convergence preserves zero entropy.

As an application, we resolve a question about the entropy of nonhyperbolic ergodic invariant measures arising from the GIKN construction proposed by Gorodetski, Ilyashenko, Kleptsyn, and Nalsky \cite{GIKN}, see also \cite{BDG,Diaz}. We call such measures \emph{royal measures}. This approach was subsequently adapted by many authors allowing them to find nonhyperbolic royal measures for systems with partially hyperbolic dynamics (see \cite{BBD, BDG, BZ, CCGWY, DGRZ, DG, KN} just to name a few). In particular, it is proved in \cite{CCGWY} that for a $\mathcal{C}^1$-generic diffeomorphism every nonhyperbolic homoclinic class carries a royal nonhyperbolic measure. Furthermore, under mild assumptions, the support of that measure is the whole class. For more on the role of the GIKN construction for finding nonhyperbolic ergodic measures, we refer the reader to Díaz's survey \cite{Diaz}. Note that another method of construction of nonhyperbolic ergodic measures with uncountable support has recently emerged \cite{BBD2, BBD3}. It uses the so-called \emph{flip-flop families} and provides a set of positive entropy supporting only nonhyperbolic measures. There are also ``mixed'' methods combining GIKN method and flip-flop families \cite{BDK,Lacka}. Again, see \cite{Diaz} for a survey and \cite{BC,DGRZ} for the latest developments.

The GIKN method is tailored to control Lyapunov exponents and ergodicity, and up to now, little was known about other properties of the resulting royal measures. The question of whether all royal measures necessarily have zero entropy has circulated in the dynamical systems community for some time (according to Lorenzo J. Díaz it was raised by Jérôme Buzzi in Orsay and the first named author learned about it from Díaz's lecture at the conference ``Global dynamics beyond uniform hyperbolicity'' in Olmué, Chile in September 2015). Although the repetitive structure inherent in the GIKN construction suggests zero entropy \cite{BBD2}, this heuristic argument has not convinced all researchers \cite{BZ}. 

Our main result establishes that royal measures always have zero entropy. The proof relies on demonstrating that the GIKN construction produces sequences of periodic points that are $\Fbar$-Cauchy (Theorem \ref{thm:GIKNmain}). Hence,  the resulting royal measures always emerge as $\Fbar$-limits of zero entropy periodic measures; even more, they are $\Fbar$-limits of measures that are ergodic and loosely Kronecker (therefore with zero entropy).  Then it follows from the properties of the Feldman-Katok convergence that royal measures necessarily inherit ergodicity (Theorem~\ref{thm:Cauchy}), zero entropy (by the lower semi-continuity of entropy, Theorem~\ref{lsc}), and the loosely Kronecker property (which also implies zero entropy) by Theorem~\ref{thm:looselyK}.

In fact, proving that royal measures are loosely Kronecker, we provide a full characterisation of the former class of measures up to Kakutani equivalence. Kakutani equivalence is a natural equivalence relation between transformations preserving an ergodic nonatomic measure that is weaker than the usual notion of isomorphism (nevertheless, the problem of deciding whether two measure preserving systems are Kakutani equivalent is in some sense equally hard as for the usual isomorphism, see \cite{GerKun}). Note that entropy is not an invariant for Kakutani equivalence, but it follows from Abramov's formula that this relation preserves the classes of zero, positive and finite, and infinite entropy transformations. Loosely Kronecker systems form the Kakutani equivalence class of any aperiodic ergodic rotation of a compact infinite group (the latter systems are known as Kronecker systems).  According to Feldman and Nadler \cite{FN}, the name loosely Kronecker systems was a suggestion of Marina Ratner. Independently, these systems were studied by Katok \cite{Katok} (partly in collaboration with Sataev \cite{KatokSataev}, who contributed to the subject as well, see~\cite{Sataev}).  
Katok called Kakutani equivalence 
\emph{monotone equivalence} and he called the loosely Kronecker systems 
\emph{standard automorphisms}. Every loosely Kronecker system has zero entropy and is loosely Bernoulli that is, it belongs to a class of measure preserving systems introduced by Feldman \cite{Feldman}. 
Recently, loosely Bernoulli (in particular, loosely Kronecker) systems have gained renewed attention. Glasner, Thouvenot, and Weiss \cite{GTW} showed that, generically, an extension of an ergodic loosely Bernoulli system is loosely Bernoulli. Gerber and Kunde \cite{GK} showed there exist smooth weakly mixing loosely Kronecker transformations whose Cartesian square is loosely Kronecker, and Trujillo \cite{Tru}  showed there exist such transformations that are mixing.

Showing that all royal measures are loosely Kronecker means that the ergodic measure-preserving system obtained from a homeomorphism and its royal invariant measure is isomorphic to a measure-preserving system arising by taking an aperiodic ergodic group rotation and the first-return transformation induced by the rotation on an appropriately chosen measurable subset of the group of positive Haar measure. This description applies to royal measures defined in \cite{BC, BBD, BDG, BZ, CCGWY, DG, GIKN, KN}.

Applications of Feldman-Katok convergence are not restricted to the study of royal measures. 
Since the  announcement of the  preprint \cite{KL-arXiv-2017} in 2017, Feldman-Katok pseudometric has been used in various contexts. 
Using Feldman-Katok pseudometric the authors of \cite{GRK} characterised uniquely ergodic dynamical systems that are loosely Kronecker. Later, Trilles \cite{Trilles} proved an analogue of this characterization for continuous flows.
Feldman-Katok pseudometric has  been also applied in studies of entropy, mean dimension, and pressure 
\cite{CaiLi-Nonlinearity-2023, XieChenYang-JDCS-2024, XieChenYang-DynamicalSystems-2025}, 
as well as it spurred investigations of dynamical pseudometrics and related notions 
\cite{CaiKwLiPourmand-JDE-2022, DKL, GaoZhang-ActaMathSin-2024}. 

\subsection*{Organisation of the paper}
Section~\ref{sec:basics} establishes the necessary preliminaries and notation, including the Feldman f-bar metric $\fbar$ for shift-invariant measures, Kakutani equivalence, and the notion of loosely Kronecker systems, which will play a central role in our classification results.
Section~\ref{sec:GIKN} presents the GIKN construction following the exposition of Bonatti, Díaz, and Gorodetski \cite{BDG}. 
We also discuss how this construction controls Lyapunov exponents to produce nonhyperbolic ergodic measures.
In Section~\ref{sec:FK}, we introduce the Feldman-Katok pseudometric $\Fbar$, extending the ideas from symbolic spaces to general compact metric spaces. We establish the basic properties of the resulting pseudometric. Section~\ref{sec:FK-convergence} develops the theory of $\Fbar$-convergence for invariant measures. We prove that $\Fbar$-Cauchy sequences of generic points determine unique invariant measures (Corollary \ref{cor:completeness}) and that $\Fbar$-convergence is stronger than weak$^*$ convergence (Theorem \ref{thm:weak}). Section \ref{sec:GIKN-FK-Cauchy} establishes that every GIKN sequence is $\Fbar$-Cauchy (Theorem \ref{thm:GIKNmain}), providing the crucial link between the geometric construction and our pseudometric framework. Section~\ref{sec:FK-ergodic} proves that $\Fbar$-limits of ergodic measures are ergodic (Theorem \ref{thm:Cauchy}) by adapting Oxtoby's criterion for ergodicity \cite{Oxtoby} to quasi-orbits. This immediately yields the ergodicity of royal measures (Corollary \ref{cor:royal-ergodic}).
While most of the results in Sections \ref{sec:FK}--\ref{sec:FK-ergodic} 
are rather easy and of technical nature, they are essential for building the Feldman-Katok pseudometric into a useful tool for studying invariant measures. This work pays dividends, as noted in the introduction, in the various subsequent applications. The main novelty comes from the ``right'' abstraction of $\fbar$ to the general case.
The main results (Theorems \ref{lsc} and \ref{thm:looselyK}) appear in Sections~\ref{sec:lsc}--\ref{sec:GIKN-LK}. 
Section~\ref{sec:lsc} demonstrates that the Kolmogorov-Sinai entropy is lower semicontinuous with respect to $\Fbar$-convergence (Theorem \ref{lsc}), implying in particular that zero entropy is preserved under $\Fbar$-limits. In Section \ref{sec:GIKN-LK}, we prove that aperiodic $\Fbar$-limits of loosely Kronecker measures are loosely Kronecker (Theorem \ref{thm:looselyK}), using Katok's criterion \cite{Katok}. Combined with our earlier results, this establishes that all royal measures are loosely Kronecker, and therefore ergodic with zero entropy and Kakutani equivalent to ergodic rotations on compact groups (Theorem \ref{thm:GIKN-improved}). This completely characterizes royal measures up to Kakutani equivalence and definitively answers the question about their entropy.

\subsection*{Acknowledgements}
First and foremost, we would like to thank Christian Bonatti, not only for many discussions related to this paper, but also
for his contagious and inspiring enthusiasm for mathematical research. 
We would like to express our equally deep gratitude to  Lorenzo J.~D\'{\i}az and Katrin Gelfert 
for their constant support, inspiration, and encouragement throughout this work. Their enthusiasm for this research motivated us to resubmit this project. We are grateful for all the many thoughtful conversations that shaped our thinking along the way and many remarks that helped us to improve this manuscript.

We would like to thank Fran\c{c}ois Ledrappier (for bringing $\fbar$-metric to our attention), Tomasz Downarowicz and Benjy Weiss (for patient and lucid answers to our multiple questions about ergodic theory), Dawid Bucki, Melih Emin Can, Philipp Kunde, Michal Kupsa, Anthony Quas, and Alexandre Trilles (for clarifications and discussions regarding $\fbar$ metric).

Special thanks to Anton Gorodetski for a very encouraging and enlightening conversation in Trieste.

The research of Dominik Kwietniak was supported by the  National Science Centre (NCN) under grant no.
DEC-2012/07/E/ST1/00185 and his stay at the Federal University of Rio de Janeiro where part of this research was conducted
is supported by CAPES/Brazil grant no. 88881.064927/2014-01. Many thanks to Maria José Pacifico who made this stay possible. 
Martha \L{}\k{a}cka acknowledges support of the National Science Centre (NCN), Poland, grant no. 2015/19/N/ST1/00872 and the doctoral scholarship no. 2017/24/T/ST1/00372. 

\section{Basic definitions and notation}\label{sec:basics} Throughout this paper $\N=\{1,2,\ldots\}$, $|A|$ is the cardinality of a set $A$, and $\chi_A$ is its characteristic function. Unless otherwise stated $i,j,k,\ell,m,n$ denote integers.

Let $\bar d(A)$ be the \emph{upper density} of a set $A\subset\N\cup\{0\}$, that is
\[\bar d(A)=\limsup\limits_{n\to\infty}\frac{|A\cap\{0,\ldots, n-1\}|}{n}.\]

The \emph{lower density} of $A\subset \Z$ is $\und{d}(A)=1-\bar{d}(\Z\setminus A)$. 
The set $A\subset\Z$ \emph{has density} $\alpha$ if $\bar{d}(A)=\und{d}(A)=\alpha$. 
Given a set $Z$ we denote by $Z^{\infty}$ the family of all $Z$-valued sequences indexed by $\N\cup\{0\}$. Typically, we write 
$\und{z}=(z_i)_{i=0}^\infty$ for elements of $Z^\infty$. By $\sigma$ we denote the shift operator acting on $Z^\infty$ as $\sigma(\und{z})=(z_{i+1})_{i=0}^\infty$.
Whenever $Z$ is a topological space, we endow $Z^\infty$ with the product topology.
\subsection{Dynamical systems (standing assumptions)} We assume that $X$ is a compact metric space, $\rho$ is a metric for $X$, and $T\colon X\to X$ is a continuous map. All results stay the same if $\rho$ is replaced by another compatible metric. 

Given $x\in X$, we distinguish between the \emph{orbit} of $x$, which is a set $\{T^n(x):n\ge 0\}\subset X$ and the \emph{trajectory} of $x$ which is a sequence
$\und{x}_T=(T^j(x))_{j=0}^\infty\in X^\infty$.

\subsection{Note on invertibility} Observe that by default we consider non-necessarily invertible transformations, but some of the results we invoke from the literature assume invertibility. It is easy to see that in all such cases, it is enough to apply the theorem we need to use the natural extension of noninvertible transformations.

\subsection{Measure-preserving systems} Most of the
standard texts on ergodic theory work with measure-preserving transformations of
standard Lebesgue spaces. The latter are measure
spaces arising as completions of probability measures on Polish metric spaces endowed with their Borel $\sigma$-algebras.
In this approach it is hard to consider different measures
on the same underlying space, since the $\sigma$-algebra depends nontrivially on the measure.
This is the primary reason we work in the Borel category. If necessary, when we work with a single invariant Borel measure $\mu$, we can take the completion of our measure space to obtain a Lebesgue space. All properties of measure-preserving systems considered here remain the same for the original system and its completion, that is, a measure-preserving system $(X,\mathscr{X},\mu,T)$ has one of these properties if and only if the completed system $(X,\tilde{\mathscr{X}}_\mu,\tilde\mu,T)$
has the property. 

\subsection{Invariant measures, generic sequences} We write $\mathscr{X}$ for the Borel $\sigma$-algebra of $X$ and $\M(X)$ for the set of all Borel probability measures on $X$. By $\MT(X)$ we denote $T$-invariant measures in $\M(X)$. We write $\MTe(X)$ for the set of ergodic invariant measures. 
The quadruple $\XX=(X,\mathscr{X},\mu,T)$ is a \emph{measure-preserving system}, which is invertible, whenever $T$ is a homeomorphism.  We give $\M(X)$ the weak$^*$ topology. Recall that $(\mu_n)_{n=1}^\infty$ converges to $\mu$ in $\M(X)$ if and only if $\int\phi\,d\mu_n\to \int \phi\,d\mu$
for every continuous $\phi\colon X\to\R$.
The weak$^*$ topology on $\M(X)$ is compact and compatible with the Prokhorov metric
\[
\di(\mu,\nu)=\inf\big\{\eps>0\,:\,\mu(B)\leq\nu(B^{\eps})+\eps\text{ for every Borel set }B\subset X\big\},
\]
where $B^\eps=\{y\in X\,:\,\dist(y,B)<\eps\}$ denotes the \emph{$\eps$-hull} of $B$.
For $x\in X$, let $\hat\delta(x)\in\mathcal M(X)$ be the Dirac measure supported on $\{x\}$. Let $\m(\und x,n)$ denote the \emph{$n$-empirical measure of $\und x=(x_j)_{j=0}^\infty\in X^{\infty}$}, that is $\m(\und x,n)=1/n(\hat\delta(x_0)+\hat\delta(x_1)+\ldots+\hat\delta(x_{n-1}))$.

Given $x\in X$ we put $\m_T(x, n)=\m(\und{x}_T,n)$. 

A measure $\mu\in \M(X)$ is \emph{generated by $\und{x}\in X^\infty$} if $\mu$ is the limit of some subsequence of $(\m(\und{x},n))_{n=1}^{\infty}$. The set of all measures generated by $\und{x}$ is denoted by $\om{\und{x}}$. We say that $\und{x}\in X^\infty$ is a \emph{generic sequence} for $\mu$, (respectively, \emph{ergodic sequence}) if $\om{\und{x}}=\{\mu\}$ for some (ergodic) 
$\mu\in\M(X)$. We write $\Gen(\mu)$  for the set of all sequences in $X^\infty$ that are generic for $\mu$.
For $z\in X$ we define $\omT{z}=\om{\und{z}_T}$, and we call $z$ a \emph{generic point} (\emph{ergodic point}) 
if $\und{z}_T$ is a generic sequence (respectively, ergodic) sequence.

Note that \emph{every} invariant measure has a generic sequence in $X^\infty$ \cite{KLO,Sigmund}, while a non-ergodic invariant measure may have no generic points. Furthermore, one can choose a generic sequence which is a \emph{quasi-orbit}. A quasi-orbit is built from long 
pieces of orbits in such a way that the set of positions at which a quasi-orbit switches from one piece of orbit to another has zero asymptotic density.
\begin{defn}
We say that $\und z=(z_n)_{n=0}^\infty\in X^\infty$ is a \emph{quasi-orbit} for $T$ if $\bar d(\{n\ge 0\,:z_{n+1}\neq T(z_n)\})=0$.
\end{defn}
It is easy to see that every measure generated by a quasi-orbit for $T$ must be $T$-invariant.

Furthermore, the GIKN construction yields a quasi-orbit generic for the invariant measure it produces. We work with that quasi-orbit to demonstrate the properties of the underlying measure.

\subsection{Symbolic systems} Let $\mathscr{A}$ be a  finite set with the discrete topology. We endow $\mathscr{A}^\infty$ with the product topology and call it the \emph{full shift} over the alphabet $\mathscr{A}$. The \emph{shift map} is the map $\sigma\colon\mathscr{A}^\infty\to \mathscr{A}^\infty$ given by 
$\sigma((x_n)_{n=0}^\infty)=(x_{n+1})_{n=0}^\infty$. The set of (ergodic) shift-invariant measures is denoted $\M_{\sigma}(\mathscr{A}^\infty)$ ($\M^e_{\sigma}(\mathscr{A}^\infty)$). 
When $\mathscr{A}=\{0,1,\ldots,k-1\}$ for some $k\in\N$ we write $\Omega_k=\{0,1,\ldots,k-1\}^\infty$. We call the elements of $\mathscr{A}^n$ \emph{words of length $n$} over $\mathscr{A}$. Let $\mathscr{A}^+=\bigcup_{n\ge 1}\mathscr{A}^n$ and $|u|$ stand for the length of $u\in\mathscr{A}^+$. Every word $u\in\mathscr{A}^+$ determines a \emph{cylinder set} $[u]\subset\mathscr{A}^\infty$ consisting of all sequences in $\mathscr{A}^\infty$, whose first $|u|$ symbols coincide with $u$. Cylinders form a clopen base for the topology of $\mathscr{A}^\infty$ and generate the Borel $\sigma$-algebra $\mathscr{B}$ of 
$\mathscr{A}^\infty$. Given two $n$-words $u=u_0u_1\ldots u_{n-1}$ and $w=w_0w_1\ldots w_{n-1}$ over $\mathscr{A}$, we define the \emph{Hamming distance} between $u$ and $w$ as
\[
\dbar_n(u,w)=\frac{1}{n}\left|\{0\le j<n:u_j\neq w_j\}\right|.
\]
The edit metric between two strings (words) of length $n$ equals $1-k/n$, where $k$ is the minimum number of symbols that have to be removed from each string so that the remaining strings are identical. In other words,
the \emph{edit distance} between $u$ and $w$ is given by
\[
\fbar_n(u,w)=1-\frac{k}{n},
\]
where $k$ is the largest among those integers $\ell$ such that for some $0\le i_1<i_2<\ldots<i_\ell<n$ and $0\le j_1<j_2<\ldots<j_\ell<n$ we have
$u_{i_s}=w_{j_s}$ for $s=1,\ldots,\ell$. For two infinite sequences ${\omega}=\omega_0\omega_1\omega_2\ldots$, ${\omega}'=\omega'_0\omega'_1\omega'_2\ldots$ in $\mathscr A^\infty$ we set
\begin{align}\label{formula}
\dbar({\omega},{\omega}')=&\limsup_{n\to\infty}\dbar_n({\omega},{\omega}')=\limsup_{n\to\infty} \dbar_n(\omega_0\omega_1\ldots \omega_{n-1},\omega'_0\omega'_1\ldots \omega'_{n-1})
 \\
 =&\dbar(\{j\ge 0:\omega_j\neq\omega'_j\}),\nonumber\\
 \label{eq:fbar}
  \fbar({\omega},{\omega}')=&\limsup_{n\to\infty}\fbar_n({\omega},{\omega}')=\limsup_{n\to\infty} \fbar_n(\omega_0\omega_1\ldots \omega_{n-1},\omega'_0\omega'_1\ldots \omega'_{n-1}),\\ 
  \label{eq:fhat}
\hat{f}({\omega},{\omega}')=&\inf\{
\eps>0:\text{ there are increasing sequences $(i_r)$, $(i_r')$ in $\N^\infty$ }    \\
&\qquad\text{of lower density at least $1-\eps$ for which $\omega_{i_r}=\omega'_{i_r'}$ for all $r\ge0$}
\}.\nonumber
\end{align}
The functions $\dbar$, $\fbar$, and $\hat{f}$ are pseudometrics on $\mathscr{A}^\infty$. Furthermore, $\fbar({\omega},{\omega}')\le \hat{f}({\omega},{\omega}')$ and $\fbar({\omega},{\omega}')\le \dbar({\omega},{\omega}')$ for $\omega,\omega'\in\mathscr{A}^\infty$ , and $\bar{f}$, $\hat{f}$ are uniformly equivalent pseudometrics on $\mathscr{A}^\infty$, see \cite{ORW}.

\subsection{Processes} We write $\bP^m(X)$ for the set of all Borel measurable partitions of $X$ into at most $m$ sets, called \emph{atoms}. For $\mathcal P\in \mathbf P^k(X)$  we write $\mathcal{P}=\{P_0,\ldots,P_{k-1}\}$ regardless of the actual number of nonempty elements in $\mathcal P$ and we agree that $P_j=\emptyset$ for $|\mathcal{P}|\le j<k$.
Let $\XX=(X,\mathscr X, \mu, T)$ be a measure-preserving system and let $\mathcal P=\{P_0,P_1,\ldots, P_{k-1}\}\in\bP^k(X)$. We identify $\mathcal P$ with a function $\mathcal P\colon X\to \{0,\ldots, k-1\}$ defined by $\mathcal P(x)=j$ for $x\in P_j$. The pair $(\XX,\mathcal P)$ is called a process, see \cite[p. 273]{Glasner}. A \emph{coding} of $\und x=(x_j)_{j=0}^{\infty}\in X^{\infty}$ is $\mathcal{P}(\und x)=(\mathcal P(x_j))_{j=0}^{\infty}\in\Omega_k$. The map $\code\colon X\to\Omega_k\text{ given by }\code(x)=\mathcal P(\und x_T)$ defines a homomorphism of $\XX$ and $(\Omega_k,\mathscr{B}, \mu_{\mathcal P}, \sigma)$, where $\mu_{\mathcal P}=\code_*(\mu)$ is the pushforward of the measure $\mu$. For $n>0$ and $\und x\in X^{\infty}$ let $\phi^n_{\mathcal P}(\und x)=\mathcal P(x_0)\mathcal P(x_1)\ldots\mathcal P(x_{n-1})\in\{0,1,\ldots, k-1\}^n$. Let $\mathcal P^n$ be the \emph{$n$-th join} of $\mathcal P$ given by
\[
\mathcal P^n=\bigvee_{j=0}^{n-1}T^{-j}(\mathcal P)=\{P_{i_0}\cap T^{-1}(P_{i_1})\cap\ldots\cap T^{-n+1}(P_{i_{n-1}}):P_{i_j}\in\mathcal P\text{ for }0\le j<n\}.
\]
Note that for $x\in X$ we may write $\phi^n_{\mathcal P}(\und{x}_T)=\mathcal P^n(x)$, because $\phi^n_{\mathcal P}(\und{x}_T)=i_0i_1\ldots i_{n-1}$ if and only if $\mathcal{P}^n(x)=P_{i_0}\cap T^{-1}(P_{i_1})\cap\ldots\cap T^{-n+1}(P_{i_{n-1}})$.
The measure-preserving system $(\Omega_k,\mathscr{B}, \mu_{\mathcal P}, \sigma)$ is called the \emph{symbolic representation} of $\XX$ with respect to the partition $\mathcal P$ and $\mu_{\mathcal P}$ is the \emph{symbolic representation measure} of $\mu$. We endow $\mathbf P^k(X)$  with the distance $d_1^\mu$ given for  $\mathcal{P},\mathcal{Q}\in\bP^k(X)$ by
\[
d^\mu_1(\mathcal{P},\mathcal{Q})=
\frac{1}{2}\sum_{j=0}^{k-1} \mu(P_j \div Q_{j})=\frac{1}{2}\sum_{j=0}^{k-1}\int_X|\chi_{P_j}-\chi_{Q_j}|\,\text{d}\mu=\mu(\{x\in X:\mathcal{P}(x)\neq\mathcal{Q}(x)\}).
\]
Note that the definition of $d^\mu_1$ takes into account the order of the partition's elements. We tacitly identify Borel partitions $\mathcal{P},\mathcal{Q}\in\bP^k(X)$ with $d^\mu_1(\mathcal{P},\mathcal{Q})=0$. With this identification $d^\mu_1$ is a complete metric for $\bP^k(X)$.

\subsection{Entropy}
For a finite measurable partition $\mathcal P$ and $\mu\in\M_T(X)$ we denote by $h(\mu, \mathcal P)$ the \emph{entropy of $\mathcal P$ with respect to $\mu$ and $T$} and by $h(\mu)$ the \emph{entropy of $\mu$ with respect to $T$}, that is $h(\mu)=\sup_{\mathcal P}h(\mu, \mathcal P)$, where $h(\mu, \mathcal P)=\inf_{n\in\N}-\sum_{P\in\mathcal P^n}\mu(P)\log\mu(P)$. The real-valued function $\mathcal P\mapsto h(\mu,\mathcal P)$ is uniformly continuous on $\bP^k(X)$ equipped with $d^\mu_1$ \cite[Lemma 15.9(5)]{Glasner}.

\subsection{Faithful coding} For $\mathcal{P}\in\bP^k(X)$ we define $\partial\mathcal{P}=\partial P_0\cup\ldots\cup\partial P_{k-1}$. A partition $\mathcal{P}\in\bP^k(X)$ with $\mu(\partial\mathcal{P})=0$ is called \emph{faithful} for $\XX$.

\begin{lem}\label{lem:generic}
If $\und{x}\in X^\infty$ is generic for $\mu\in\MT(X)$ and $\mathcal{P}\in\bP^k(X)$ is such that $\mu(\partial\mathcal{P})=0$, then $\omega=\mathcal{P}(\und x)\in\Omega_k$ is a 
generic point for the measure $\mu_{\mathcal{P}}$ on $\Omega_k$. 
\end{lem}
\begin{proof}
Note the following two properties of the boundary operator $\partial$: $\partial(Y\cap Z)\subset \partial Y\cup\partial Z$ for $Y,Z\subset X$ and $\partial T^{-1}(U)\subset T^{-1}(\partial U)$ for any $U\subset X$. 
Using this and $\mu(\partial P_j)=0$ for every $0\le j < k$ we see that
\[
\mu(\partial(\bigcap_{i=0}^{m-1}T^{-i}(P_{j_i}))=0\quad\text{for every }m\ge 1\text{ and }0\le j_0,j_1,\ldots,j_{m-1}< k.
\]
In other words, $\mu(\partial\mathcal{P}^m)=0$ for all $m\ge 1$.  Then for every $m\ge 1$ and $0\le j_0,j_1,\ldots,j_{m-1}< k$ we have
\[
\lim_{N\to\infty}\frac{1}{N}\sum_{j=0}^{N-1}\chi_A(x_j)=\mu(A),\quad \text{for }A=\bigcap_{i=0}^{m-1}T^{-i}(P_{j_i})\in\mathcal{P}^m.
\]
Note that $\omega=\mathcal{P}(\und x)$ is an orbit for $\sigma$ and observe that $\omega$ is generic for a $\sigma$-invariant measure $\mu'$ such that
\[
\mu'([j_0j_1\ldots j_{m-1}])=\mu(\bigcap_{i=0}^{m-1}T^{-i}(P_{j_i})) \quad\text{for every }m\ge 1\text{ and }0\le j_0,j_1,\ldots,j_{m-1}< k.
\]
Hence $\mu'$ and $\mu_{\mathcal{P}}$ agrees on cylinders in $\Omega_k$. This implies $\mu'=\mu_{\mathcal{P}}$.
\end{proof}

\subsection{Kakutani equivalence} Kakutani equivalence serves as a natural equivalence relation among transformations that preserve an ergodic nonatomic measure, weaker than the conventional concept of isomorphism (however, in some sense equally hard, cf. \cite{GerKun}). Although entropy does not remain invariant under Kakutani equivalence, Abramov's formula implies that this relation retains the categories of zero, positive and finite, as well as infinite entropy transformations.
Let $\XX=(X,\mathscr X, \mu, T)$ be a measure-preserving system. For a set $E\in \mathscr X$ with $\mu(E)>0$ and $x\in E$ we define the \emph{return time} $n(x)=\inf\{k>0: T^k(x)\in E\}$. This function is finite for $\mu$-almost every $x\in E$ and we define the \emph{induced transformation} $T_E\colon E\to E$ by $T_E(x)=T^{n(x)}(x)$. Measure-preserving systems $\XX=(X,\mathscr X, \mu, T)$  and $\mathbf Y=(Y,\mathscr Y, \nu, S)$ are \emph{Kakutani equivalent} (recall, that Katok calls this relation  \emph{monotone equivalence}) if there exist $E\in\mathscr X$ with $\mu(E)>0$ and $F\in\mathscr Y$ with $\nu(F)>0$ such that $T_E$ is isomorphic with $S_F$.

\subsection{Loosely Bernoulli systems}
Feldman \cite{Feldman} studied the isomorphism problem in ergodic theory. He introduced a new property (called \emph{loose Bernoulliness}) for finite partitions of a measure-preserving system $(X,B,\mu,T)$. He used it to construct important examples of  $K$-automorphisms which are loosely Bernoulli, but not Bernoulli. In particular, these measure-preserving systems have positive entropy.
The definition of loosely Bernoulli partition  follows Ornstein's definition \cite{Ornstein} of very weak Bernoulli partition with the Hamming distance of strings of symbols of length $n$ replaced by the weaker edit metric $\fbar_n$.  Feldman's idea was subsequently extended by Ornstein, Rudolph, and Weiss \cite{ORW}. It turns out that, in each of three entropy classes, the Kakutani equivalence class of loosely Bernoulli transformations is the simplest one, see \cite{ORW} for more details. Positive and finite (respectively, infinite) entropy transformations Kakutani equivalent to a Bernoulli shift coincide with loosely Bernoulli and positive and finite (respectively, infinite) entropy transformations.
Zero entropy loosely Bernoulli transformations form the Kakutani equivalence class of any ergodic rotation of a compact infinite group (Kronecker system).  According to Feldman and Nadler \cite{FN}, members of the latter equivalence class are called \emph{loosely Kronecker} following a suggestion of Marina Ratner. We adapt this terminology and provide a characterization of these systems in Section \ref{sec:losselyK}, cf. Theorem \ref{criterion}.
\subsection{Feldman's $\fbar$ and Ornstein's $\dbar$ metrics for shift-invariant measures on $\mathscr A^\infty$}\label{sec:fbar-dbar}
Let $\mu$ and $\mu'$ be ergodic shift invariant measures on $\mathscr{A}^\infty$. By $\mu_n$, respectively $\mu'_n$ we denote the restriction of $\mu$, respectively $\mu'$ to the set of all $n$-cylinders, that is, these are measures that $\mu$, respectively $\mu'$, define on $\mathscr{A}^n$ via the projections onto first $n$ coordinates. A \emph{joining} of $\mu$ and $\mu'$ is any $\sigma\times \sigma$-invariant measure on $\mathscr{A}^\infty\times \mathscr{A}^\infty$  whose marginals are $\mu$ and $\mu'$. We write $J(\mu,\mu')$ for the set of all joinings of $\mu$ and $\mu'$. Similarly,
$J_n(\mu,\mu')$ denotes the set of all measures $\lambda_n$ on $\mathscr{A}^n\times \mathscr{A}^n$ whose marginals are $\mu_n$ and $\mu'_n$.

Define
\begin{align}\label{eq:fbar-for-measures}
\bar{f}_n(\mu,\mu')=&\inf_{\lambda_n\in J_n(\mu,\mu')}\int_{\mathscr{A}^n\times \mathscr{A}^n}\bar{f}_n(u,u') \lambda_n(u,u'),\\
\label{eq:fbar-metric-for-measures}
\bar{f}(\mu,\mu')
=&\limsup_{n\to\infty}\bar{f}_n(\mu,\mu').
  \end{align}
One can prove (see \cite{ORW}) that \eqref{eq:fbar-metric-for-measures} defines a distance between measures $\mu$ and $\mu'$ on $\mathscr A^\infty$ known as the \emph{f-bar} metric. 
Ornstein's \emph{d-bar} metric $\dbar$ on $\M_\sigma(\mathscr A^\infty)$ is defined analogously with the Hamming distance $\dbar_n$ on $\mathscr A^n$ replacing the edit distance $\fbar_n$ in \eqref{eq:fbar-for-measures}, see \cite{Shields}.
\subsection{Properties of $\fbar$}
For the readers' convenience, we include here some statements extracted from \cite{ORW}. We rephrase them in a way suitable for our purposes.
The first result is a direct corollary of Propositions 2.6 and 2.7 in \cite{ORW}.
\begin{lem}\label{lem:fbar-bound}
For every $\eps>0$ there is $\delta>0$ such that if $\mu$ and $\mu'$ are shift-invariant ergodic measures on $\mathscr A^\infty$
and there are generic points $\omega$ for $\mu$ and $\omega'$ for $\mu'$
with $\fbar(\omega,\omega') < \delta$, then $\fbar( \mu, \mu' ) < \eps$.
\end{lem}

Finally, note that the entropy function $\mu\mapsto h(\mu)$ is uniformly continuous with respect to the $\fbar$-metric on the space of shift-invariant measures  on $\mathscr A^\infty$
(for ergodic measures this is \cite[Proposition 3.4]{ORW}, the assumption of ergodicity can be removed thanks to the main result of \cite{DKL}).

\subsection{Loosely Kronecker systems} \label{sec:losselyK} An ergodic measure-preserving system $\XX=(X,\mathscr{X},\mu,T)$ is \emph{loosely Kronecker} if it has zero entropy and for every finite Borel partition $\mathcal{P}$ of $X$ and every $\eps>0$ there are $n>0$ and a set $A_n$ of atoms of $\mathcal{P}^n=\bigvee_{j=0}^{n-1}T^{-j}(\mathcal{P})$ such that
$\mu(A_n)>1-\eps$ and $\fbar_n(u,w)<\eps$ for $u,w\in A_n$ (here, as usual, we identify atoms of the partition $\mathcal{P}^n$ with words of length $n$ over the alphabet $\{0,1,\ldots,|\mathcal{P}|-1\}$).

\section{GIKN construction}\label{sec:GIKN}

We present the construction of Gorodetski, Ilyashenko, Kleptsyn, and Nalsky from \cite{GIKN} following the exposition provided by Bonatti, D\'{\i}az, and Gorodetski in~\cite{BDG}. Recall that $X$ is a compact metric space, $\rho$ is a metric for $X$, and $T\colon X\to X$ is a continuous map. By definition, each royal measure is the weak$^*$ limit of measures supported on a sequence of periodic points with special properties. Each periodic orbit in this sequence can be divided in two parts: the \emph{shadowing part} and the \emph{tail}. Their key features
are:
\begin{enumerate}
  \item The shadowing part takes a large proportion (growing to $1$ as $n$ goes to $\infty$) of each periodic orbit in the sequence. Furthermore, the images of each point on the shadowing part of the $n$-th orbit are $\gamma_n$-close to the $(n-1)$-th orbit for the number of iterates equal to the primary period of the $(n-1)$-th orbit and the series formed by $\gamma_n$'s is summable. This is used to show that the limit measure is ergodic.
  
  \item There is a fixed (center) direction such that the Lyapunov exponent in that direction along the sequence of periodic orbits decreases to zero. To achieve this, each periodic orbit spends a small proportion of its primary period (this part of the orbit is called the \emph{tail}) in a region which is far from the previous orbits and is chosen so that the Lyapunov exponent in the center direction along the whole orbit has smaller absolute value than the same exponent for the previously constructed orbits. This also guarantees that the limit measure is nonatomic.
\end{enumerate}

\subsection{Topological backbone of the GIKN construction} In this paragraph we sketch the main features of the GIKN construction. Our results about $\Fbar$-limits like Theorems \ref{lsc} and \ref{thm:looselyK} apply to any measure defined in this way, because the GIKN construction leads to $\Fbar$-convergent sequences of periodic orbits, s Theorem \ref{thm:GIKNmain}.
\begin{defn}\label{def:good-approx}
We say that a $T$-periodic orbit $\Gamma$ is a \emph{$(\gamma,\kappa)$-good approximation} of a $T$-periodic orbit $\Lambda$ if there are a subset $\Delta$ of $\Gamma$ with $|\Delta|/|\Gamma|\ge\kappa$ and a constant-to-one surjection $\psi\colon\Delta\to\Lambda$ (called \emph{$(\gamma,\kappa)$-projection}) such that for each $y\in\Delta$ and $0\le j <|\Lambda|$ we have
\[
\rho(T^j(y),T^j(\psi(y)))<\gamma.
\]
\end{defn}
\begin{defn}\label{def:GIKN-sequence}
We call a sequence $(\Gamma_n)_{n\in\N}$ of $T$-periodic orbits with $|\Gamma_n|\nearrow\infty$ as $n\to\infty$ 
a \emph{GIKN sequence} if
there are sequences of positive real numbers $(\gamma_n)_{n=1}^\infty$ and $(\kappa_n)_{n=1}^\infty$ such that the following conditions hold:
\begin{enumerate}
  \item for each $n\ge1$, 
  $\Gamma_{n+1}$ is a $(\gamma_n,\kappa_n)$-good approximation of $\Gamma_n$,
  \item $\sum_{n=1}^\infty\gamma_n<\infty$,
  \item $\prod_{n=1}^\infty\kappa_n>0$. 
\end{enumerate} 
By an abuse of terminology, we will also refer to the sequence of ergodic measures $(\mu_n)_{n=1}^\infty$, where $\mu_n$ is  supported on $\Gamma_n$, as to \emph{GIKN sequence of measures}.
\end{defn}
The following is a slightly reformulated \cite[Lemma 2.5]{BDG}. The proof that $\mu$ is ergodic in~\cite{BDG} invokes \cite[Lemma 2]{GIKN}.
\begin{thm}\label{thm:GIKN}
Let $(\Gamma_n)_{n\in\N}$ be a GIKN sequence of $T$-periodic orbits. If $(\mu_n)_{n=1}^\infty$ is the associated sequence of ergodic measures,  
then $(\mu_n)_{n\in\N}$ weak$^*$ converges to an ergodic measure $\mu$ supported on the topological limit of $(\Gamma_n)_{n\in\N}$, that is,
\[
\supp\mu=\bigcap_{k=1}^\infty\overline{\bigcup_{n\ge k}\Gamma_n}.
\]
\end{thm}

\begin{defn}
A measure $\mu$ is \emph{royal} if it is a weak$^*$ limit of a GIKN-sequence.
\end{defn}

With this terminology, Theorem \ref{thm:GIKN} states that every GIKN-sequence determines a \emph{royal measure}. 

\subsection{Lyapunov exponents} We discuss how the GIKN construction is used  to find nonhyperbolic measures. This part is logically independent from the rest of the paper. 

Let $M$ be a smooth Riemannian manifold with $\dim M=m$. If $f\colon M\to M$ is a diffeomorphism and $\mu$ is an ergodic $f$-invariant measure, then there exist a set $\Lambda\subset M$ of full $\mu$-measure and real numbers $\chi^1_\mu\le\ldots\le \chi^m_\mu$ such that for every $x\in \Lambda$ and nonzero vector $v\in T_xM$ one has
\[
\lim_{n\to\infty}\frac{1}{n}\log\|Df_x^n(v)\|=\chi^i_\mu\quad\text{for some }i=1,\ldots,m.
\]
The number $\chi^i_\mu$ is the $i$th \emph{Lyapunov exponent} of the measure $\mu$. If there exists a closed $f$-invariant set $\Xi\subset M$ and a continuous $Df$-invariant  direction field $\mathcal E=(E_x)_{x\in\Xi}\subset T_\Xi M$ with $\dim E_x=1$ for $x\in\Xi$, then for every measure $\nu\in\M_f(M)$ with $\supp\nu\subset\Xi$ there is a Lyapunov exponent $\chi^{\mathcal E}(\nu)$ of $\nu$ associated with $\mathcal E$ in the following sense: for $\nu$-a.e. $x\in M$ and a nonzero vector $v\in E_x$ one has
\[
\lim_{n\to\infty}\frac{1}{n}\log\|Df_x^n(v)\|=\chi^{\mathcal E}(\nu).
\]
Furthermore, if $(\mu_n)_{n\in\N}$ is a sequence 
in $\M^e_f(\Xi)$, $\mu\in\M^e_f(M)$ and $\mu_n\to \mu$ as $n\to\infty$ in the weak$^*$ topology (it implies that $\supp\mu\subset\Xi$ as well), then $\chi^{\mathcal E}(\mu_n)\to\chi^{\mathcal E}(\mu)$ as $n\to\infty$ \cite[Lemma 1]{GIKN}.

\begin{thm}\label{thm:BDG}
Assume that $f\colon M\to M$ is a diffeomorphism of a smooth Riemannian compact manifold with a closed $f$-invariant set $\Xi\subset M$ and a continuous $Df$-invariant  direction field $\mathcal E=(E_x)_{x\in\Xi}\subset T_\Xi M$ with $\dim E_x=1$ for $x\in\Xi$. Let $(\Gamma_n)_{n\in\N}\subset\Xi$ be a sequence of $f$-periodic orbits and suppose that $|\Gamma_n|$ increases to infinity as $n\to\infty$. For each $n$ let $\mu_n$ be the ergodic measure supported on $\Gamma_n$. Furthermore, assume that the following conditions hold:
\begin{enumerate}
  \item There exist sequences of positive real numbers $(\gamma_n)_{n=1}^\infty$ and a constant $C$ such that for each $n$ the orbit $\Gamma_{n+1}$ is a $(\gamma_n,1-C|\chi^{\mathcal E}(\mu_n)|)$-good approximation of $\Gamma_n$;
  \item There exists a constant $0<\alpha<1$ such that \[|\chi^{\mathcal E}(\mu_{n+1})|<\alpha|\chi^{\mathcal E}(\mu_n)|;\]
  \item $\displaystyle\gamma_n<\frac{\min_{1\le i\le n}d_i}{3\cdot 2^n}$, where $d_i$ denotes the minimal distance between different points in $\Gamma_i$.
\end{enumerate}
Then $(\Gamma_n)_{n\in\N}\subset\Xi$ is a GIKN sequence and $(\mu_n)_{n\in\N}$ weak$^*$ converges to a royal measure $\mu$ that is ergodic and has uncountable support equal to the topological limit of $(\Gamma_n)_{n\in\N}$. 
Furthermore, $\mu$ is nonhyperbolic, since $\chi^{\mathcal E}(\mu)=0$.
\end{thm}

\section{Feldman-Katok Pseudometric}\label{sec:FK}

In this section, we first recall the definition of Besicovitch pseudometric, then introduce Feldman-Katok pseudometric, which extends $\fbar$ to general metric spaces.

For symbolic dynamical systems (subsystems of the full shift $(\mathscr{A}^\infty,\sigma)$) the metrics $\dbar$ and $\fbar$ on $\M_\sigma(\mathscr A^\infty)$ defined in Section \ref{sec:fbar-dbar} are related with identically denoted pseudometrics on $\mathscr A^\infty$ given by, respectively,  \eqref{formula} and \eqref{eq:fbar}.
This connection is described in more detail in~\cite{ORW, Shields, Weiss}. From the point of view of ergodic theory there is no need to extend $\dbar$ and $\fbar$ from $\mathscr A^\infty$ to more general metric spaces. It turns out, however, that for some geometric applications an extension of $\dbar$, called the \emph{Besicovitch pseudometric}, is very useful (see \cite{BCKO,BL,KLO} and references therein).  On the other hand, we are not aware of any analogue of $\fbar$ pseudometric for general metric space in the literature.

\subsection{Besicovitch pseudometric $D_B$}
For $\und x=(x_j)_{j=0}^\infty,\,\und z=(z_j)_{j=0}^\infty\in X^{\infty}$ we define the \emph{Besicovitch pseudometric} $D_B$ on $X^\infty$ as
\[D_B(\und x, \und z)=\limsup\limits_{n\to\infty}\frac{1}{n}\sum_{j=0}^{n-1}\rho(x_j, z_j).\]
It is known \cite{KLO} that $D_B$ is uniformly equivalent to $D_B'$ given by
\begin{equation}
    \label{eq:DBprime}
    D'_B(\und{x},\und{z})=\inf
\{\delta>0: \dbar(\{n\ge 0: \rho(x_n,z_n)\ge \delta\})<\delta\}.
\end{equation}
Given $T\colon X\to X$, the \emph{Besicovitch pseudometric} $D_B$ on $X$ is defined by $D_B(x,y)=D_B(\und{x}_T,\und{z}_T)$. If $X=\mathscr A^\infty$ and $\rho$ is any metric compatible with the topology on $\mathscr A^\infty$, then the Besicovitch pseudometric $D_B$ and the d-bar pseudometric $\dbar$ are uniformly equivalent on $\mathscr A^\infty$.

\subsection{Pseudometric $\Fbar$} 
Fix $\und x=(x_j)_{j=0}^\infty$, $\und z=(z_j)_{j=0}^\infty\in X^{\infty}$, $\delta>0$, and $n\in \N$.
\begin{defn}\label{def:match}
An \emph{$(n,\delta)$-match} between $\und x$ and $\und z$ is an order-preserving bijection $\pi\colon\Dom(\pi)\to\Ran(\pi)$ such that $\Dom(\pi),\Ran(\pi)\subset \{0,1,\ldots,n-1\}$ and for every $i\in\Dom(\pi)$ we have $\rho(x_i,z_{\pi(i)})<\delta$. The \emph{fit} $|\pi|$ of  $\pi$ is the cardinality of $\Dom(\pi)$.
\end{defn}
\begin{defn}
An $(n,\delta)$-match is  \emph{maximal} if its fit is the largest possible. If there is no $(n,\delta)$-match $\pi$ with $|\pi|\ge 1$, then the \emph{empty match $\pi_\emptyset$} with $|\pi_\emptyset|=0$ is the maximal one. The \emph{$(n,\delta)$-gap} between $\und x$ and $\und z$ is given by
\[
\fbar_{n,\delta}(\und x,\und z)=1-\frac{\max\{|\pi|:\text{ $\pi$ is an $(n,\delta)$-match of $\und x$ with $\und z$}\}}{n}.
\]
\end{defn}

Note that if $X=\mathscr{A}^\infty$ for some finite alphabet $\mathscr{A}$ and we endow $X=\mathscr{A}^\infty$ with the standard metric given by
\[
\rho(\und\omega,\und\omega')=\begin{cases}
                               0, & \mbox{if }\und\omega=\und\omega',  \\
                               2^{-\min\{j\ge0: \omega_j\neq\omega'_j\}}, & \mbox{otherwise},
                             \end{cases}
\]
then (with a minor abuse of notation) we have $\fbar_{n,1}(\und x,\und z)=\fbar_n(x_0x_1\ldots x_{n-1},z_0z_1\ldots z_{n-1})$ for any
$\und x,\und z\in\mathscr{A}^\infty$.

We note some properties of the $(n,\delta)$-gap function $\fbar_{n,\delta}$ that hold for any
$\und x$, $\und z\in X^{\infty}$, $\eps,\delta>0$, and $n\in \N$. 
The proofs are left to the reader.

\begin{fact}\label{fact:fbar_mono}
If $0<\delta<\delta'$, then $\fbar_{n,\delta'}(\und x,\und z)\le\fbar_{n,\delta}(\und x,\und z)$. 
\end{fact}

\begin{fact}\label{fact:fbar_n+q}
If 
$q\in \N$, then $\fbar_{n,\delta}(\und x,\und z)\le\fbar_{n+q,\delta}(\und x,\und z)+q/n$. 
\end{fact}
\begin{fact}\label{fact:fbar_implies_di}
If $\fbar_{n,\delta}(\und x,\und z)<\eps$, then 
$\di(\m( \und x, n), \m( \und z, n))< \max\{\delta,\eps\}$.
\end{fact}

\begin{defn}
The \emph{$\fbar_\delta$-pseudodistance} between $\und x$ and $\und z$ is given by
\[
\fbar_{\delta}(\und x,\und z)=\limsup_{n\to\infty}\fbar_{n,\delta}(\und x,\und z).
\]
\end{defn}

\begin{fact}\label{fact:fbar_Ti}
If 
$i\in\N\cup\{0\}$, then $\fbar_\delta(\und x,\und z)=\fbar_\delta(\und x,\sigma^i(\und z))$. 
In particular, $\fbar_\delta(\und x,\sigma^i(\und x))=0$.
\end{fact}

\begin{fact}\label{fact:superad}
Assume that $\und x$ and $\und z$ are periodic with a common period $N$. If $p^{(n)}$ denotes the fit of a maximal $(nN,\delta)$-match, then the sequence $(p^{(n)})_{n=1}^\infty$ is subadditive, that is $p^{(\ell)}+p^{(m)}\ge p^{(\ell+m)}$ for every $\ell,m\in\N$.
\end{fact}

\begin{fact}\label{fact:f_delta_for_per}
If $\und x$ and $\und z$ are periodic sequences with a common period $N$, then
\[
\fbar_{\delta}(\und x,\und z)=\inf_{n\in\N}\fbar_{nN,\delta}(\und x,\und z)=\lim_{n\to\infty}\fbar_{nN+q,\delta}(\und x,\und z)
\]
for every $0\le q<N$.
\end{fact}
\begin{defn}
The \emph{Feldman-Katok pseudometric on $X^{\infty}$} is given by
\[
\Fbar(\und{x},\und{z})=\inf\{\delta>0:\fbar_{\delta}(\und x,\und z)<\delta\}.
\]
The \emph{Feldman-Katok pseudometric on $X$}, also denoted by $\Fbar$, is defined for $x,z\in X$ by
\[\Fbar(x,z)=\Fbar\big(\und{x}_T,\und{z}_T\big).\]
The symbols $\f_{n,\delta}(x,z)$ and $\f_{\delta}(x,z)$ have the obvious meaning.
\end{defn}

\begin{fact} \label{fact:fbar_d+e}
If $\und x,\und z\in X^{\infty}$ and $\fbar_\delta(\und x,\und z)\le \eps$, for some $\delta,\eps>0$, then $\Fbar(\und x,\und z)\le \delta+\eps$.
\end{fact}
\begin{proof}
If $\fbar_\delta(\und x,\und z)\le \eps$, then $\fbar_{\delta+\eps}(\und x,\und z)<\delta+\eps$ by Fact~\ref{fact:fbar_mono}. Thus $\Fbar(\und x,\und z)\le \delta+\eps$.
\end{proof}
\begin{rem} By Fact \ref{fact:fbar_d+e} the set $\{\delta>0:\fbar_{\delta}(\und x,\und z)<\delta\}$ is nonempty for every $\und x,\und z\in X^{\infty}$.
Furthermore, $0\le \fbar_\delta(\und x,\und z)\le 1$ for all $\und x,\und z\in X^{\infty}$ and $\delta>0$. This together with Fact \ref{fact:fbar_d+e} imply that $0\le \Fbar(\und x,\und z)\le 1$ for every $\und x,\und z\in X^{\infty}$.
\end{rem}

\begin{fact}
The function $\Fbar$ is a pseudometric on $X^\infty$, as well as on $X$.
\end{fact}
\begin{defn}
We say that $x,y\in X$ are \emph{orbitally related} and write $x\stackrel{T}{\sim}y$ if $T^i(x)=T^j(y)$ for some $i,j\ge 0$.
\end{defn}
\begin{fact}\label{fact:orbequi}
If $x\stackrel{T}{\sim}x'$ and $y\stackrel{T}{\sim}y'$, then $\Fbar(x,y)=\Fbar(x',y')$.
\end{fact}
\begin{proof} It is enough to prove that $\fbar_\delta(x,y)=\fbar_\delta(x',y')$ for every $\delta>0$. Thus we fix $\delta>0$.
If $x\stackrel{T}{\sim}x'$ and $y\stackrel{T}{\sim}y'$, then there are $i,j,m,n\ge 0$ such that $T^i(x)=T^j(x')$ and $T^m(y)=T^n(y')$. Using Fact \ref{fact:fbar_Ti} repeatedly we have
\[
\fbar_\delta(x,y)=\fbar_\delta(T^i(x),y)=\fbar_\delta(T^j(x'),y)=\fbar_\delta(x',y),
\]
and similarly, $\fbar_\delta(x',y)=\fbar_\delta(x',y')$. 
\end{proof}
By Fact \ref{fact:orbequi} the Feldman-Katok pseudometric depends rather on the separation between forward orbits of given points, than of the points alone.

\subsection{Comparison with the Besicovitch Pseudometric} We note that our results about Feldman-Katok pseudometric generalize those known for Besicovitch pseudometric (see \cite{BCKO, BJK, KLO} for more details). In particular, any sequence of generic points converging in the Besicovitch pseudometric provides and example of $\Fbar$-convergent sequence of measures.
\begin{lem}\label{lem:DBFbar}
If $\und x,\und z\in X^{\infty}$, then $\Fbar(\und x, \und z)\leq D_B'(\und x, \und z)$, where $D_B'$ is given by \eqref{eq:DBprime}.
\end{lem}
\begin{proof} Define
\[
\dbar_{n,\delta}(\und x,\und z)=\frac{1}{n}|\{0\le j< n : \rho(x_j, z_j)\ge\delta\}|.
\]
If $D_B'(\und x, \und z)<\delta$ for some $\delta>0$ then for all $n$ large enough $\dbar_{n,\delta}(\und x, \und z)<\delta$. It follows that there exists an $(n, \delta)$-match of $\und x$ with $\und z$. Therefore $\Fbar(\und x, \und z)\leq D_B'(\und x, \und z)$.
\end{proof}

\section{$\Fbar$-convergence of measures and its properties}\label{sec:FK-convergence}
The relationship between $\Fbar$ and some metric on either $\MTe(X)$ or $\MT(X)$ is not clear, while the pseudometric $\fbar$ defined on $\mathscr A^\infty$ in \eqref{eq:fbar} relates to a metric on $\M_{\sigma }(\mathscr A^\infty)$ given by \eqref{eq:fbar-metric-for-measures}. Nonetheless, using $\Fbar$ we can still introduce a form of ``convergence'' in $\MT(X)$.

\begin{defn}
We say that a sequence of measures $(\mu_n)_{n=1}^\infty\subset\MT(X)$ \emph{converges in $\Fbar$} or \emph{$\Fbar$-converges} to $\mu\in\MT(X)$ if
there exists a sequence of quasi-orbits $(\und x^{(n)})_{n=1}^\infty\subset X^\infty$ with $\om{\und x^{(n)}}=\{\mu_n\}$ such that for some $\mu$-generic quasi-orbit $\und z\in X^\infty$ 
we have $\Fbar(\und z,\und{x}^{(n)})\to 0$ as $n\to\infty$.
\end{defn}

Keeping in mind that we are interested in the GIKN construction, we will examine properties of the $\Fbar$-convergence in a special case, where the quasi-orbits $(\und x^{(n)})_{n=1}^\infty\subset X^\infty$ are actually orbits (for each $n\in\N$, $\und x^{(n)}$ is the orbit of a $\mu_n$-generic point).

\subsection{``Completeness'' of $\Fbar$-convergence} Since $\Fbar$ is a pseudometric on $X^\infty$, given $T\colon X\to X$ such notions as \emph{$\Fbar$-Cauchy sequence} or \emph{$\Fbar$-limit} have obvious meaning. The first important feature of the $\Fbar$-convergence is that it has enough ``completeness'' for our purposes: an $\Fbar$-Cauchy sequence of orbits defines a quasi-orbit which is its $\Fbar$-limit.
\begin{defn}
We say that a sequence of quasi-orbits $(\und x^{(n)})_{n=1}^\infty\subset X^\infty$ is $\Fbar$-Cauchy if for every $\eps>0$ there is $N\in\N$ such that
$\Fbar(\und x^{(k)},\und x^{(\ell)})<\eps$ for $k,\ell\ge N$.
\end{defn}
Repeating the proof of \cite[Proposition 2]{BFK} we get the following fact.

\begin{lem}\label{lem:completness}
If $(x^{(n)})_{n=1}^\infty\subset X$ is an $\Fbar$-Cauchy sequence on $X$, then there
is a quasi-orbit $\und z=(z_j)_{j=0}^\infty\in X^\infty$ such that
$\Fbar(\und{x}^{(n)}_T,\und z)\to 0\text{ as }n\to\infty$. 
\end{lem}

In fact, the proof of Lemma \ref{lem:completness} yields that there are $0=m_0<m_1<m_2<\ldots$
and a subsequence $(x^{(n_j)})_{j=1}^\infty$ of $(x^{(n)})_{n=1}^\infty$ such that
$\dbar(\{m_j:j\ge 0\})=0$ and $\und z$ is given by $z_n=T^n
(x^{(n_j)})$ for $m_{j-1}\le n <m_{j}$.

We do not know whether under the assumptions of Lemma \ref{lem:completness} there is a point whose orbit is a $\Fbar$-limit of the $\Fbar$-Cauchy sequence. We can prove it only by imposing an extra assumption on $T\colon X\to X$.  It turns out that the \emph{asymptotic average shadowing property} is a sufficient requirement. Let us recall its definition. 

\begin{defn} A sequence $\und{z}=(z_j)_{j=0}^\infty\in X^{\infty}$ is \emph{asymptotic average pseudoorbit} for $T\colon X\to X$ if
\[
D_B(T(\und{z}),\und{z})=\lim\limits_{N\to\infty}\frac{1}{N}\sum_{j=0}^{N-1} \rho(T(z_j),z_{j+1})=0.
\]
We say that $\und z=(z_j)_{j=0}^\infty\in X^{\infty}$ is \emph{asymptotically shadowed in average} by $x\in X$, if
\begin{equation}\label{eq:AASP}
D_B(\und{x}_T,\und{z})=\lim\limits_{N\to\infty}\frac{1}{N}\sum_{j=0}^{N-1}\rho\big(T^j(x),z_j\big)=0.    
\end{equation}
A system $(X,T)$ has the \emph{asymptotic average shadowing property} if every asymptotic average pseudo-orbit of $T$ is asymptotically shadowed in average by some point. 
\end{defn}

Note that if $\und{x}_T$ and $\und{z}$ satisfy \eqref{eq:AASP}, then $D_B'(\und{x}_T,\und{z})=\Fbar(\und{x}_T,\und{z})=0$ as well, see \eqref{eq:DBprime} and Lemma \ref{lem:DBFbar}. The asymptotic average shadowing property was introduced by Gu \cite{Gu} and is implied by most of the versions of the specification property considered in the literature, see \cite{CT,KKO,KLO2,KLO}.

\begin{cor}
If $(X,T)$ satisfies the asymptotic average shadowing property, then $\Fbar$ is a~complete pseudometric on $X$.
\end{cor}
\begin{proof}
Fix an $\Fbar$-Cauchy sequence $(x^{(n)})_{n=1}^\infty\subset X$. Note that the quasi-orbit $\und z\in X^{\infty}$ provided by Lemma \ref{lem:completness} is an asymptotic average pseudoorbit and satisfies $\Fbar(\und{x}_T,\und z)=0$. 
Pick $x\in X$ which asymptotically shadows in average $\und z$. Then
\[
\Fbar(x,x^{(n)})\leq \Fbar(\und{x}_T,\und z)+\Fbar(\und z,\und{x}^{(n)}_T)\to 0\quad\text{as }n\to\infty.
\]
Therefore $x$ is an $\Fbar$-limit of $(x^{(n)})_{n=1}^\infty$ in $X$.
\end{proof}

It turns out that the set of generic quasi-orbits is $\Fbar$-closed. More is true: $\om{\und x}$ depends $\Fbar$-continuously on
$\und x$ when we consider $\om{\und x}$ as a point in the space of nonempty closed subsets of $\M(X)$ endowed with the
Hausdorff metric $D_H$ induced by the Prokhorov metric $D_P$ (the latter space is called the hyperspace of $\M(X)$).


\begin{fact}\label{H}
If $\Fbar(\und x,\und z)<\eps$ for some 
$\eps>0$, then $D_H(\om{\und x},\om{\und z})<\eps$, where $D_H$ is the Hausdorff metric on the hyperspace of $\M(X)$ endowed with $D_P$.
\end{fact}
\begin{proof}
If $\Fbar(\und x,\und z)<\eps$, then for some $\Fbar(\und x,\und z)\le \delta <\eps$ we have
$\fbar_\delta(\und x,\und z)<\delta<\eps$. Therefore $\fbar_{n,\delta}(\und x,\und z)<\delta$ for all $n$ large enough and
$\di(\m(\und x, n), \m(\und z, n))< \eps$ by Fact \ref{fact:fbar_implies_di}.
This implies that 
$D_H(\om{\und x},\om{\und z})<\eps$. 
\end{proof}

Although the proof of the following fact is short, its importance justifies calling it a theorem.

\begin{thm}\label{thm:weak}
Let $(\und x^{(n)})_{n=1}^\infty\subset X^{\infty}$ be such that for each $n\in\N$ there is $\mu_n\in\M(X)$
with $\und{x}^{(n)}\in\Gen(\mu_n)$. If $\und x\in X^{\infty}$ and $\Fbar(\und x,\und x^{(n)})\to 0$ as $n\to\infty$, then there exists $\mu\in\M(X)$ such that $\mu_n\to\mu$ as $n\to\infty$ in $\M(X)$ and $\und x\in\Gen(\mu)$.

\end{thm}
\begin{proof}
By Fact~\ref{H}, $\Fbar(\und x,\und x^{(n)})\to 0$ as $n\to\infty$ implies that $\om{\und x^{(n)}}\to\om{\und x}$ as $n\to\infty$ in the hyperspace of $\M(X)$. 
Since $\und x^{(n)}$ is generic for $\mu_n$, we have $\om{\und x^{(n)}}=\{\mu_n\}$ for $n\in\N$. The family of all singletons is closed in the hyperspace and homeomorphic with $\M(X)$. Thus $\om{\und x}$ must also be a singleton, that is $\om{\und x}=\{\mu\}$ for some $\mu\in\M(X)$ and $\mu_n\to\mu$ as $n\to\infty$ in
$\M(X)$.
\end{proof}
An immediate consequence of Lemma \ref{lem:completness} and Theorem \ref{thm:weak} is the following.
\begin{cor}\label{cor:completeness}
An $\Fbar$-Cauchy sequence of generic points uniquely determines an invariant measure.
\end{cor}
Corollary \ref{cor:completeness} allows us to justify the correctness of the next definition. 
\begin{defn}\label{def:fbar-lim}
An invariant measure determined by a $\Fbar$-Cauchy sequence of generic points is called the \emph{$\Fbar$-limit} of the corresponding sequence of measures.
\end{defn}

\section{Every GIKN sequence is $\Fbar$-Cauchy}\label{sec:GIKN-FK-Cauchy}

All examples of $(\gamma,\kappa)$-good approximations found in the literature support the conjecture that each pair of periodic orbits satisfying Definition \ref{def:good-approx} is always $(\gamma+(1-\kappa))$-close with respect to $\Fbar$. However, the definition of $(\gamma,\kappa)$-good approximation  does not ensure that the $(\gamma,\kappa)$-projection $\psi$ appearing in Definition \ref{def:good-approx} is a match as in Definition \ref{def:match}. This is because a match must be order-preserving, and Definition \ref{def:good-approx} does not stipulate this requirement from a projection. Therefore, the following technical lemma is necessary.

\begin{lem}\label{lem:GIKNCauchy}
If a $T$-periodic orbit $\Gamma$ is a $(\gamma,\kappa)$-good approximation of a $T$-periodic orbit $\Lambda$, then for any choice of $x\in\Gamma$ and $z\in\Lambda$ one has $\Fbar(x,z)<\gamma+(1-\kappa)$.
\end{lem}
\begin{proof}
It follows from Fact~\ref{fact:orbequi} that it is enough to find $x_0\in\Gamma$ and $z_0\in\Lambda$ such that $\Fbar(x_0,z_0)<\gamma+(1-\kappa)$.

Let $\psi\colon\Delta\to\Lambda$ be the $(\gamma,\kappa)$-projection from $\Gamma$ to $\Lambda$.
Pick any $x_0\in\Delta$ and let $z_0=\psi(x_0)\in\Lambda$.
By Fact \ref{fact:fbar_d+e} it is enough to show that $\fbar_\gamma(x_0,z_0)<1-\kappa$. Define $q=|\Lambda|$.
Furthermore, since $x_0$ and $z_0$ are periodic, we conclude from Fact \ref{fact:f_delta_for_per} and Fact \ref{fact:fbar_n+q} that it is sufficient to find for any multiple $p$ of $|\Gamma|$ and $|\Lambda|$ a $(p+q,\gamma)$-match $\pi$ of $x_0$ with $z_0$ such that $|\pi|\ge \kappa p$. Let $p$ be a multiple of $|\Gamma|$ and $|\Lambda|$. For $j\ge 1$ define $x_j=T^j(x_0)$ and $z_j=T^j(z_0)$. Note that $(z_j)_{j=0}^\infty$ is a $q$-periodic sequence. We will abuse the notation and treat $\{x_0,\ldots,x_{p-1}\}$, $\{z_0,\ldots,z_{p-1}\}$ as sets with $p$ elements, still denoted by $\Gamma$ and $\Lambda$. Furthermore we extend $\psi$ to a function from $\Delta\subset \{x_0,\ldots,x_{p-1}\}$ to $\{z_0,\ldots,z_{p-1}\}$ where $|\Delta|=r$ and $r/p\ge \kappa$.



We are going to define the match $\pi$ by performing an inductive construction with at most $|\Delta|$ steps. At each step we will extend the domain of $\pi$ by at least one element. Enumerate elements of $\Delta$ as $y_0,\ldots,y_{r-1}$ in such a way that the order induced by this indexing of $\Delta$ coincides with the one induced by enumerating the elements of $\Gamma$ as $x_0,\ldots,x_{p-1}$. By definition $x_0=y_0$.
Let $\theta\colon \{0,\ldots,r-1\}\to\{0,1,\ldots,p-1\}$ be a function such that $\theta(s)$ for $0\le s< r$ is the position of $y_s$ in  the sequence $x_0,\ldots,x_{p-1}$, that is, $\theta(s)=j$ if and only if $y_s=x_j$.
We begin with $\Dom(\pi)=\{0\}$ and $\pi(0)=0$. 
By definition, $z_{\pi(0)}=\psi(x_0)=z_0$, hence $\rho(x_0,z_{\pi(0)})<\gamma$, and $\pi$ is a $(p+q,\gamma)$-match.

Now assume that we have already performed some number of steps of our construction and we have obtained a $(p+q,\gamma)$-match $\pi$
such that $0\le s < r$ is the largest integer satisfying: $|\pi|\ge s+1$, $\theta(s)\in\Dom(\pi)$, $\pi(\theta(s))\le \theta(s)$ and $z_{\pi(\theta(s))}=\psi(y_s)$. It follows from our construction that $s+1$ is always greater than or equal to the number of steps performed.
We can extend $\pi$ by setting $\pi(\theta(s)+i)=\theta(s)+i$ 
for $0< i <q$. 
Now, there are two cases to consider: either the domain of $\pi$ contains $\{\theta(0),\ldots,\theta(r-1)\}$ and we are done or there is $s< t< r$ which is the smallest integer such that $\pi$ is not defined at $\theta(t)$. Clearly, in the latter case $\theta(t)-\theta(s)\ge q$. Since
$(z_j)_{j=0}^\infty$ is a $q$-periodic sequence there exists $1\le \ell \le q$ such that $z_{\theta(s)+\ell}=\psi(y_t)$, and $\pi(\theta(s))<\theta(s)+\ell\le \theta(t)$. We want to set $\pi(\theta(t))=\theta(s)+\ell$, but to keep $\pi$ increasing we need to remove first at most $q-1$ elements from $\Dom(\pi)$ as defined so far, namely those in $\Dom(\pi)\cap\{\theta(s)+\ell,\ldots,\theta(s)+q-1\}$.  But then we can set $\pi(\theta(t)+i)=\theta(t)+i=\theta(s)+\ell+i$ for $1\le i <q$, because
\[
z_{\pi(\theta(t)+i)}=T^i(z_{\theta(t)})\text{ and }\rho(T^i(z_{\theta(t)}),T^i(\psi(y_t)))<\gamma.
\]
For our new $\pi$ we see that the largest integer $s$ in $\{0,1,\ldots,r-1\}$ satisfying $|\pi|\ge s+1$, $\theta(s)\in\Dom(\pi)$, $\pi(\theta(s))\le \theta(s)$ and $z_{\pi(\theta(s))}=\psi(y_s)$ is larger than or equal to $t$. Thus in a finite number of steps our procedure will produce a $(p+q,\gamma)$-match $\pi$ with $|\Dom(\pi)|\ge\kappa p$.
\end{proof}
As a consequence we see that a royal measure 
is the $\Fbar$-limit of periodic measures, which immediately implies \cite[Lemma 2.5]{BDG} except conclusions about ergodicity and support, see also \cite[Lemma 2]{GIKN}.

\begin{thm}\label{thm:GIKNmain}
If $(\Gamma_n)_{n\in\N}$ is a GIKN-sequence of $T$-periodic orbits and  $\mu_n$ denotes the ergodic measure supported on $\Gamma_n$ for $n\in\N$, then for any choice $x_n\in\Gamma_n$ the sequence $(x_n)_{n=1}^\infty$ is $\Fbar$-Cauchy and $\mu_n$ $\Fbar$-converges to some $\mu$.
In particular, a royal measure is a $\Fbar$-limit of periodic measures.
\end{thm}

\section{$\Fbar$-limit of ergodic measures is ergodic}\label{sec:FK-ergodic}

We are going to show that an $\Fbar$-limit of a sequence of ergodic measures must be ergodic. To this end we first present a criterion for ergodicity of a measure generated by a quasi-orbit. We obtain it by an easy adaptation of an analogous criterion for orbits given by Oxtoby \cite{Oxtoby}.

\subsection{Auxiliary terminology and results} For $k\in\N$, $\und x=(x_j)_{j=0}^\infty\in X^{\infty}$ and $\phi\in\mathcal{C}(X)$ let $A_k(\phi,\und x)$ denote the Birkhoff average along $x_0,x_1,\ldots,x_{k-1}$, that is,
\[
A_k(\phi,\und x)=\frac{1}{k}\sum_{j=0}^{k-1}\phi(x_j)=\int_X \phi\,d\m(\und x,k).
\]
For $x\in X$ we write $A_k(\phi,x)$ for the Birkhoff average along orbit segment of length $k$, that is $A_k(\phi,x)=A_k(\phi,\und{x}_T)$.
Recall that a sequence $\und x\in X^{\infty}$ (a point $x\in X$, respectively) is generic for some measure $\mu$ if and only if for every $\phi\in\mathcal C(X)$ the sequence $A_k(\phi,\und x)$ (respectively, $A_k(\phi,x)$) converges as $k\to\infty$. We denote the corresponding limit by $\phi^*(\und x)$ (respectively, $\phi^*(x)$). It is easy to see that for any $\ell\in\N$ we have $\phi^*(\und x)=\phi^*(\sigma^\ell(\und x))$, respectively $\phi^*(x) =\phi^*(T^\ell(x))$.
For a generic sequence $\und x$ we put
\[
A^*_k(\phi,x_\ell)=|A_k(\phi,(x_j)_{j=\ell}^{\infty})-\phi^*((x_j)_{j=\ell}^{\infty})|=|A_k(\phi,\sigma^\ell(\und x))-\phi^*(\und x)|.
\]
Furthermore, for a generic point $x$ and $\ell\in\N$ we define
\[
A^*_k(\phi,T^\ell(x))=|A_k(\phi,T^\ell(x))-\phi^*(x)|.
\]

Fix $\alpha>0$, $k\in\N$, and $\phi\in\mathcal{C}(X)$. We say that $\ell\ge 0$ is a \emph{starting point of an $(\alpha,\phi)$-bad $k$-segment} in
a generic sequence $\und{x}=(x_j)_{j=0}^\infty$ if $A_k^*(\phi,\sigma^\ell(\und{x}))>\alpha$.

The following characterization of ergodic sequences slightly generalizes the one presented by Oxtoby in~\cite[Section 4]{Oxtoby}.
It states that a generic sequence $\und{x}=(x_j)_{j=0}^\infty$ generates an ergodic measure if for every $\phi\in\mathcal{C}(X)$ and $\alpha>0$ the upper density of the set of integers $\ell$ being starting points for an $(\alpha,\phi)$-bad $k$-segments converges to zero as $k$ goes to $\infty$. Replacing $\und{x}$ by $\und{x}_T$ we obtain a criterion for ergodicity of a measure given by a generic point due to Oxtoby. We omit the proof as it follows the same lines as in~\cite{Oxtoby}.

\begin{thm}[Oxtoby's Criterion]\label{thm:Oxtoby}
Let a quasi-orbit $\und z=(z_j)_{j=0}^\infty\in X^{\infty}$ be generic for some $\mu\in\M_T(X)$. 
Then $\mu$ is ergodic if and only if for every $\phi\in\mathcal{C}(X)$ and $\alpha>0$ we have
\[\bar d\big(\{\ell\ge 0\,:\,|A_k(\phi,\sigma^\ell(\und z))-\phi^*(\und z)|>\alpha\}\big)\to 0\text{ as }k\to\infty.\]
\end{thm}

Let $n\in\N$, $x\in X$ and $\und z\in X^{\infty}$. For an $(n,\delta)$-match $\pi$ of $\und{x}_T$ with $\und{z}$ and $k\le n$ we define for every $\ell\in\Dom(\pi)\cap\{0,1,\ldots,n-k\}$ a set
\[
\Dom'_\ell=
\{i\in\Dom(\pi): \ell\le i < \ell+k \text{ and }\pi(\ell)\le \pi(i) < \pi(\ell)+k
\}
\]
and the \emph{$(k,\delta)$-match $\pi'_\ell\colon \Dom'_\ell\to\pi(\Dom'_\ell)$
induced by $\pi$ at $\ell$} setting $\pi'_\ell(i)=\pi(i)$ for $i\in\Dom(\pi'_\ell)$. It is easy to see that
$\pi'_\ell$ is indeed an $(k,\delta)$-match of $T^\ell(x)$ with $z_{\pi(\ell)}$ satisfying
$\Dom(\pi'_\ell)=\Dom'_\ell$ and $\Ran(\pi'_\ell)=\pi(\Dom'_\ell)$.

We also need the following technical lemma, which asserts that if an orbit $\und x_T$ and a quasi-orbit $\und z$ are sufficiently $\Fbar$-close, then
for any $\phi\in\mathcal C(X)$ and $k\in\N$ one can find a match $\pi$ which allows one to find a match $\pi$ of $\und x_T$ and $\und z$ so that the averages of $\phi$ over $k$ segments in $\und x_T$ and $\und z$ are also small for most of $k$-segments.
\begin{lem}\label{mmm+}
Fix $\phi\in\mathcal C(X)$ and $\eps>0$. Let $\delta>0$ be such that $y,y'\in X$ and $\rho(y,y')<\delta$ imply $|\phi(y)-\phi(y')|<\eps$.
If $\und z=(z_j)_{j=0}^\infty\in X^{\infty}$ is 
a quasi-orbit and $x\in X$ satisfies $\Fbar(\und{x}_T,\und z)<\delta$, then for every $k\in\N$ there exists $N\in\N$ such that for every $n\geq N$ there are an $(n,\delta)$-match $\pi$ of $\und{x}_T$ with $\und z$ and a set $A\subset\mathcal D(\pi)$ with $|A|>n(1-2\sqrt{\delta}-2\delta)-k$ satisfying
\[\big|A_k(\phi, T^\ell(x))-A_k(\phi,\sigma^{\pi(\ell)}(\und{z}))\big|\leq \eps+4\sqrt\delta||\phi||_{\infty}\quad\text{for every }\ell\in A.\]
\end{lem}
\begin{proof}
Fix $k\in\N$, $\phi\in\mathcal C(X)$ and $\eps>0$. Choose $N\in\N$ such that for every $n\geq N$ there exists
an $(n,\delta)$-match $\pi$ of $\und{x}_T$ with $\und z$ satisfying $|\pi|>n(1-\delta)$ and
\[|\{0\leq j<n\,:T(z_{j+i})\neq z_{j+i+1}\text{ for some }0\leq i<k\}|<n\delta.\]
Fix $n\geq N$. Define
\begin{align*}
A_Z&=\{0\leq j<n\,:T(z_{\pi(j)+i})\neq z_{\pi(j)+i+1}\text{ for some }0\leq i<k\},\\
A_R&=\Big\{0\leq j<n-k\,:\,j\in\Dom(\pi)\text{ and }\big|\{0\leq i<k\,:\,\pi(j)+i\notin\Ran(\pi)\}\big|\geq\sqrt\delta k\Big\},\\
A_D&=\Big\{0\leq j<n-k\,:\,j\in\Dom(\pi)\text{ and }\big|\{0\leq i<k\,:\,j+i\notin\Dom(\pi)\}\big|\geq\sqrt\delta k\Big\}.\end{align*}
Note that $|A_Z|\le n\delta$. To estimate $|A_R|$, define $\Ran^{c}(\pi)=\{0,1,\ldots,n-1\}\setminus\Ran(\pi)$, and note that
\[
|\Ran^c(\pi)|\ge |\{m\in \Ran^c(\pi)\,:\,\pi(j)\le m <\pi(j)+k\text{ for some }j\in A_R\}|.
\]
On the other hand, for each $m\in\Ran^c(\pi)$ the set $\{j\in A_R\,:\,\pi(j)\le m <\pi(j)+k\}$ has at most $k$ elements, hence each $j\in A_R$ implies that there are at least $\sqrt{\delta}$ members of $\Ran^c(\pi)$. Thus $|\Ran^c(\pi)|\ge|A_R| \sqrt{\delta}$. This together with $|\Ran^c(\pi)|<n\delta$ gives us $|A_R|< n\sqrt\delta$. An analogous reasoning leads to $|A_D|< n\sqrt\delta$.
Define \[A=\big(\{0,\ldots, n-k-1\}\cap \mathcal D(\pi)\big)\setminus(A_Z\cup A_D\cup A_R).\]
Then $|A|> n(1-2\sqrt\delta-2\delta)-k$ and for every $j\in A$ we have
\begin{enumerate}[(i)]
\item the $(k,\delta)$-match $\pi'_j$ induced by $\pi$ at $j$ satisfies $|\pi'_j|>(1-2\sqrt{\delta})k$,
\item $A_k(\phi,z_{\pi(j)})=A_k(\phi,\sigma^{\pi(j)}(\und{z}))$.
\end{enumerate}
Therefore
\begin{multline*}
\big|A_k(\phi, T^j(x))-A_k(\phi,\sigma^{\pi(j)}(\und{z}))\big|=\big|A_k(\phi, T^j(x))-A_k(\phi,z_{\pi(j)})\big|\leq\\\leq\frac 1k\sum_{i\in \Dom(\pi'_j)}|\phi(T^{i}(x))-\phi(z_{\pi(i)})|+\frac{1}{k}\sum_{i\notin \Dom(\pi'_j)}|\phi(T^{i}(x))|+\frac 1k\sum_{i\notin \Ran(\pi'_j)}|\phi(z_i)|\leq\\\leq \eps+2\sqrt\delta||\phi||_{\infty}+2\sqrt\delta||\phi||_{\infty}=\eps+4\sqrt\delta||\phi||_{\infty}
\end{multline*}
and the lemma follows.
\end{proof}

\subsection{$\Fbar$-limits of ergodic measures are ergodic}

We prove the main theorem of this section. As the GIKN sequence of periodic orbits is $\Fbar$-Cauchy by Theorem \ref{thm:GIKNmain} we see that our Theorem \ref{thm:Cauchy} generalizes \cite[Lemma 2]{GIKN}. Furthermore, since $\Fbar\le D_B$ this result extends also \cite[Theorem 15]{KLO}.
\begin{thm}\label{thm:Cauchy}
If $(x^{(p)})_{p=1}^{\infty}\subset X$ is an $\Fbar$-Cauchy sequence of ergodic points 
then it determines an ergodic measure.
\end{thm}
\begin{proof}
By Fact~\ref{H} there exists a measure $\mu\in\mathcal\M_T(X)$ such that $\mu_p\to\mu$ as $p\to\infty$ in the weak$^*$ topology on $\M_T(X)$. We will apply Oxtoby's criterion (Theorem \ref{thm:Oxtoby}) to show that $\mu$ is ergodic. Let $\und z=(z_j)_{j=0}^\infty$ be a quasi-orbit
such that $\f(\und{x}^{(p)}_T,\und z)\to 0$ as $p\to\infty$ 
provided by Lemma \ref{lem:completness}. Clearly, $\und z$ is generic for $\mu$. Fix $\phi\in\mathcal{C}(X)$ and $\alpha>0$. We need to show that for every $\eta>0$ and all sufficiently large $k$ the set of $j$'s with initiating an $(\alpha,\phi)$-bad $k$-segment in $\und{z}$ has upper density smaller than $\eta$.

Note that for every $i,j,k,p\in\N$ we have 
\begin{multline}
\label{ineq:main}
|A_k(\phi,\sigma^j(\und{z})-\phi^*(\und z)|\leq |A_k(\phi,\sigma^j(\und{z}))-A_k(\phi,T^i(x^{(p)}))|+\\+|A_k(\phi,T^i(x^{(p)}))-\phi^*(x^{(p)})|+\left|\phi^*(x^{(p)})-\phi^*(\und z)\right|.
\end{multline}
By Theorem~\ref{thm:weak} and Fact \ref{H} we 
can choose $P(\alpha)\in\N$ such that for every $p\geq P(\alpha)$ one has
$\left|\phi^*(x^{(p)})-\phi^*(\und z)\right|\leq\alpha/3$.

Let $\delta_0>0$ be such that $2\sqrt\delta_0+2\delta_0<\eta/4$ and $4\sqrt\delta_0||\phi||_{\infty}<\alpha/6$. Pick $\delta<\delta_0$ such that for all $y,y'\in X$ with $\rho(y,y')<\delta$ one has $|\phi(y)-\phi(y')|<\alpha/6$. Note that our choice of constants is motivated by Lemma \ref{mmm+}.
Choose $p\geq P(\alpha)$ such that $\fbar(\und z,x^{(p)}_T)<\delta$.
Pick $K>0$ such that for all $k\geq K$ the upper density of the set of $i$'s initiating an $(\alpha/3,\phi)$-bad $k$-segment in $x^{(p)}_T$ is smaller than $\eta/2$. We are going to prove that each $k\ge K$ is ``sufficiently large'' to imply that the upper density of
the set of $j$'s initiating an $(\alpha,\phi)$-bad $k$-segment in $\und{z}$ is smaller than $\eta$.  To this end fix any $k\ge K$.
Let $N$ be sufficiently large to guarantee that $k/N<\eta/4$ and for every $n\ge N$ we have
\begin{enumerate}[(i)]
\item\label{1} there exists an $(n,\delta)$-match $\pi$ of $\und z$ with $x^{(p)}$ such that $|\pi|>(1-\delta)n$,
\item\label{2} $|\{0\leq j<n\,:T(z_{j+i})\neq z_{j+i+1}\text{ for some }0\leq i<k\}|<n\delta$,
\item\label{3} $|\{0\le i<n\,: \text{$i$ initiates an $(\alpha/3,\phi)$-bad $k$-segment in $x^{(p)}_T$}\}|< n\eta/2$.
\end{enumerate}

We are going to show that for each $n\ge N$ the number of $0\le j<n$ such that
\[
A^*_k(\phi,\sigma^j(\und{z}))=|A_k(\phi,(z_s)_{s=j}^{\infty})-\phi^*(\und z)|\leq\alpha
\]
is larger than $(1-\eta)n$. To this end let $\tilde\pi$ be any extension of $\pi$ to a bijection from $\{0,\ldots,n-1\}$ onto $\{0,\ldots,n-1\}$.
It follows from the proof of Lemma~\ref{mmm+} that conditions~\eqref{1} and~\eqref{2} guarantee that the number of integers $0\leq j<n$ for which
\[
|A_k(\phi,\sigma^j(\und{z}))-A_k(\phi,T^{\tilde\pi(j)}(x^{(p)}))|
<\alpha/6+4\sqrt\delta_0||\phi||_{\infty}<\alpha/3
\]
is larger than $n(1-2\sqrt\delta-2\delta)-k$ and hence larger than $n(1-\eta/2)$. 
Some of these $j$'s may initiate an $(\alpha/3,\phi)$-bad $k$-segment in $x^{(p)}_T$, but \eqref{3} bounds from the above the number of such $j$'s by $n\eta/2$. Therefore setting $i=\tilde\pi(j)$ in \eqref{ineq:main} we see that
the number of integers $0\leq j<n$ for which all summands on the right hand side of \eqref{ineq:main} are bounded above by $\alpha/3$  is larger than $n(1-\eta)$. Since this holds true for any $n\ge N$ we see that upper density of the set of all $j$'s initiating an $(\alpha,\phi)$-bad $k$-segment in $\und{z}$ is smaller than $\eta$, hence the Oxtoby criterion yields that $\und z$ is an ergodic sequence.
\end{proof}

Ergodicity of royal measures is now a straightforward corollary, hence \cite[Lemma 2.5]{BDG} is recovered except the statement about the support (which can be now proved in the same way as in~\cite{BDG}, since the proof of that part of \cite[Lemma 2.5]{BDG} is independent from the rest).
\begin{cor}\label{cor:royal-ergodic}
Every royal measure is ergodic.
\end{cor}

\section{Lower semicontinuity of entropy in $\Fbar$}\label{sec:lsc}
We are going to show that the function assigning to a generic point the entropy of the associated measure is lower semicontinuous with respect to $\Fbar$.  \begin{thm}\label{lsc}
If a sequence of measures $(\mu_n)_{n=1}^\infty$ converges in $\Fbar$ to $\mu_0\in\MT(X)$, then
\[
h(\mu_0)\le \liminf_{n\to\infty} h(\mu_n).
\]
\end{thm}
Before proceeding with the proof, we first note the following technical result allowing us, given an $\Fbar$-converging sequence of invariant measures to replace any partition with a faithful partition without much change in $\fbar$-distance between codings of generic points of the measures in the sequence.
\begin{lem}\label{lem:repartition}
If $\delta>0$, $\mathcal{P}\in\bP^k(X)$ 
and $(\mu_j)_{j=0}^{\infty}\subset\M_T^e(X)$, then
there are $\gamma>0$ and 
$\mathcal{R}\in\bP^{k+1}(X)$ with $d^{\mu_0}_1(\mathcal{P},\mathcal{R})<\delta$ 
satisfying: if $\und x\in X^\infty$ is generic for $\mu_0$, $j\in\N$, and $\und z\in X^\infty$ is generic for  $\mu_{j}$ with $\Fbar(\und x,\und z)<\gamma$,  
then $\fbar(\mathcal{R}(\und x),\mathcal{R}(\und z))<\delta+\Fbar(\und x,\und z)$ 
and $\mu_j(\partial\mathcal{R})=\mu_0(\partial\mathcal{R})=0$. The same holds if we replace $\fbar$ by $\dbar$ and $\Fbar$ by $D_B$.
\end{lem}

\begin{proof} Fix $\delta>0$. Using regularity of $\mu_0$ for each $1\le j\le k$ we can find a compact set
$R_j\subset P_j$ such that $\mu_0(P_j\setminus R_j)<\delta/(8k^2)$. For $1\le i < j\le k$ the sets $R_i$ and $R_j$ are compact and disjoint, hence
\[
\Delta=\min\{\rho(x,y):x\in R_i,\,y\in R_j,\,1\le i < j\le k\}>0.
\]
Set $\tilde\gamma=\Delta/2$. For each $1\le j\le k$ and $0\le c<\tilde\gamma$  let $R^c_j$ be the closed $c$-hull around $Q_j$, that is,
$R^c_j=\{x\in X: \dist(x,Q_j)\le c\}$ and set
$R^c_0=X\setminus (R^c_1\cup\ldots\cup R^c_k)$. For each $0\le c<\tilde\gamma$ define the partition $\mathcal{R}^c=\{R^c_0,R^c_1,\ldots,R^c_k\}$. It is easy to see that $d_1^{\mu_0}(\mathcal{P},\mathcal{R}^\alpha)<\delta/2$ and $d_1^{\mu_0}(\mathcal{R}^\alpha,\mathcal{R}^\beta)<\delta/2$ for any $0\le \alpha,\beta<\tilde\gamma$.

For $1\le j\le k$ and $0<c<\tilde\gamma$ define a set $\partial_c Q_j =\{x\in X: \dist (x,Q_j)=c\}$.
Note that $\partial_c Q_j$ contains (but need not to be equal) the topological boundary of the set $R^c_j=\{x\in X: \dist (x,Q_j)\le c\}$.
Consider a family of closed sets $\mathcal{C}=\{\partial_c Q_1\cup\ldots\cup\partial_c Q_k: 0<c<\tilde\gamma\}$.
Since elements of $\mathcal{C}$ are pairwise disjoint, only countably many of them can have positive $\mu_j$ measure for some $j\in\N\cup\{0\}$. Therefore the set $E$ of all parameters $0<c<\tilde\gamma$ such that for each $c\in E$ the set $\partial_c Q_1\cup\ldots\cup\partial_c Q_k\in\mathcal{C}$ is a $\mu_j$-null set for $j\in\N\cup\{0\}$ has at most countable complement in $(0,\tilde\gamma)$.
Thus we can pick $\alpha,\beta\in E$ with $\alpha<\beta$ and $\beta-\alpha>\tilde\gamma/2$. Set $\gamma=\tilde\gamma/2$. Let $\und z\in X^\infty$ be a generic sequence for $\mu_j$ for some $j\in\N$ and $\fbar(\und x,\und z)<\gamma$. Note that $\beta-\alpha>\fbar(\und x,\und z)$. Define $x''=\mathcal{R}^\alpha(\und x)$, $z'=\mathcal{R}^\beta(\und z)$ and $x'=\mathcal{R}^\beta(\und x)$ (all three points are considered as elements of the shift space over the alphabet $\{0,1,\ldots,k\}$).

We claim that $\fbar(x'',z')< \fbar(\und x,\und z)+\delta/2$. Indeed, if for some $m,n\ge 0$ and $1\le j\le k$ we have $x_n\in R^\alpha_j$  and $\rho(x_n,z_m)<\fbar(\und x,\und z)<\beta-\alpha$, then $\mathcal{R}^\beta(z_m)=j$. Furthermore, genericity of $\und x$ for $\mu_0$ and $\mu_0(\partial R^\alpha_0)=0$ implies that $d(\{n\ge 0: x_n\in R^\alpha_0\})=\mu_0(R^\alpha_0)<\delta/2$. This proves the claim.

Using again that $\und x$ is generic for $\mu_0$, and $\partial\mathcal{R}^\alpha$, $\partial\mathcal{R}^\beta$ are $\mu_0$-null we have easily that $\dbar(x',x'')\le d_1^{\mu_0}(\mathcal{R}^\alpha,\mathcal{R}^\beta)$.
This together with the claim above and $\fbar(x',x'')\le \dbar(x',x'')$ complete the proof for the $\fbar$ case. The $\dbar$-part follows the same way.
\end{proof}

\begin{proof}[Proof of Theorem \ref{lsc}] 
Fix $\eps>0$. Let $\mathcal{P}=\{P_1,\ldots,P_k\}$ be a measurable partition of $X$ with $h(\mu_0,\mathcal{P})\ge h(\mu_0)-\eps/3$.
Let $\zeta>0$ be so small that $\f(y,y')<\zeta$ for two shift generic points $y,y'\in\{0,1,\ldots,k\}^\infty$ implies that the entropies of the corresponding measures differ by at most $\eps/3$ (the existence of such a $\zeta$ is guaranteed by \cite[Proposition 3.4]{ORW}, \cite{DKL}).
Since the function $\mathcal P\mapsto h(\mu,\mathcal P)$ is uniformly continuous on on $\bP^{k+1}$ \cite[Lemma 15.9(5)]{Glasner}, we may pick $0<\delta<\zeta/2$ such that for any partition $\mathcal{S}=\{S_0,S_1,\ldots,S_k\}$ with $d_1^{\mu_0}(\mathcal{P},\mathcal{S})<\delta$ we have $|h(\mu_0,\mathcal{P})-h(\mu_0,\mathcal{S})|<\eps/3$.

Use Lemma~\ref{lem:repartition} to find $\gamma>0$ for $\mu_0$, $\delta/2$ and $\mathcal{P}$. Let $N\in\N$ be such that $\Fbar(\und x,\und{x}^{(N)}_T)<\min\{\gamma,\delta/2\}$.
Take the partition $\mathcal{R}$ provided by Lemma~\ref{lem:repartition} for $\und x$, $\und{x}^{(N)}_T$ and $\mu_N$. Let $x'=\mathcal{R}(\und x)$ and $x'_N=\mathcal{R}(x^{(N)}_T)$.
It follows from Lemma~\ref{lem:repartition} that
$x'_N$ is a generic point for some measure $\mu'_N$ with
$h(\mu'_N)=h(\mu_N,\mathcal{R})\le h(\mu_N)$. By Lemma we see that $\fbar(x',x'_N)<\delta$.  Therefore $|h(\mu'_0)-h(\mu'_N)|<\eps/3$, hence
\[
h(\mu_N)\ge h(\mu_N,\mathcal{S})=h(\mu'_N)\ge h(\mu_0)-\eps.
\]
This finishes the proof.
\end{proof}

Note that by \cite{DKL} entropy is uniformly continuous if $X=\mathscr{A}^\infty$ and we equip $\M_{\sigma}(\mathscr{A}^\infty)$ with $\fbar$ metric (which, as we noted above, is uniformly equivalent with $\Fbar$ induced by a standard metric on $\mathscr{A}^\infty$). 
This is also the case, when the entropy function $\MT(X)\ni \mu\to h(\mu)\in\R$ is upper semicontinuous with respect to the weak$^*$ topology on $\MT(X)$. For example, this happens if $X$ is a manifold and $T$ is of class $\mathcal{C}^\infty$ or, more generally, if $T$ is asymptotically $h$-expansive, see \cite{Downarowicz}. 
\begin{cor}\label{cor:continuity}
If $(X,T)$ is such that the function $\MT(X)\ni \mu\to h(\mu)\in\R$ is upper semicontinuous with respect to the weak$^*$ topology on $\MT(X)$, then for every $\eps>0$ there is $\delta>0$ such that for $\mu,\mu'\in\MT(X)$ whenever there are $\und x\in\Gen(\mu)$ and $\und x'\in\Gen(\mu')$ with $\Fbar(\und x,\und x')<\delta$, then $\lvert h(\mu)-h(\mu')\rvert<\eps$.
\end{cor}

Without extra assumptions, the entropy function need not be continuous with respect to the $\Fbar$-convergence.


\begin{ex} Let $\mathscr{A}=\{0\}\cup\{1/k:k\in\N\}$ with the topology inherited from $[0,1]$. Consider $X=\mathscr{A}^\infty$ with any metric compatible with the product topology. Let $\xi^{(n)}$ be a measure on $\mathscr{A}$ uniformly distributed on $\{1/\ell: 2^n\le \ell <2^{n+1}\}$ and $\mu^{(n)}$ denote the product measure on $X$. It is easy to see that $\mu^{(n)}$ is a shift invariant measure on $\mathscr{A}^\infty$ and the sequence $(\mu^{(n)})_{n\in\N}$ converges in $\Fbar$ to the measure concentrated on a $\sigma$-fixed point $(0,0,\ldots)\in\mathscr A^\infty$. Furthermore, $h(\mu^{(n)})=n\log 2$, which means that the entropy function cannot be upper semicontinuous.
\end{ex}

\section{GIKN construction leads to loosely Kronecker measure}\label{sec:GIKN-LK}

We are going to show that $\Fbar$-limit of loosely Kronecker measure 
is either a periodic measure or a loosely Kronecker measure. Since the GIKN construction yields a measure with an uncountable support, every royal measure must be loosely Kronecker.

\begin{thm}\label{thm:looselyK}
An aperiodic $\Fbar$-limit of periodic measures (loosely Kronecker measures) is loosely Kronecker.
\end{thm}

We present the proof at the end of this section. Before that we recall Katok's criterion for standardness (loosely Kronecker) and formulate two technical lemmas we will need for the proof.

\subsection{Katok's criterion} Let $\XX=(X,\mathscr X, \mu, T)$ be a measure-preserving system and let $\mathcal P=\{P_0,P_1,\ldots, P_{k-1}\}\in\bP^k(X)$. Following Katok \cite[Definition 9.1]{Katok} we say that the process $(\XX, \mathcal P)$ is \emph{$(n,\eps)$-trivial} if there exists a word $\omega\in\{0,1,\ldots, k-1\}^n$ such that $\mu_{\cP}(B_{\eps}[\omega])\geq 1-\eps$, where $B_{\eps}[\omega]=\big\{\omega'\in\{0,1,\ldots, k-1\}^n\,:\,\fbar_{n}(\omega, \omega')<\eps\big\}$. 
For every $n\ge 1$, $\eps>0$ and $\omega\in\{0,1\ldots,k-1\}^n$ we clearly have
\[
(1-\mu_{\cP}(B_{\eps}[\omega]))\eps\le 
\int_X \fbar_n(\cP^n(x),\omega)\text{ d}\mu(x)\le 
(1-\mu_{\cP}(B_{\eps}[\omega]))+\eps.
\]
Therefore, (cf. \cite[Lemma 9.1]{Katok}) 
if $\omega\in\{0,1,\ldots,k-1\}^n$ and $\beta>0$ are such that
\[
\int_X \fbar_n({\cP}^n(x),\omega)\,\text{d}\mu(x)=\int_{\Omega_k} \fbar_n(u_0u_1\ldots u_{n-1},\omega)\,\text{d}\mu_\mathcal{P}(u)<\beta,
\]
then the process $(\XX, \mathcal P)$ is $(n, \sqrt\beta)$-trivial. Conversely, if the process $(\XX, \mathcal P)$ is $(n,\eps/2)$-trivial, then  
there is $\omega^{(n)} \in\mathcal{P}^n$ such that 
\[\int_X \fbar_n(\cP(x),
\omega^{(n)})\text{ d}\mu(x)=\int_{\Omega_k} \fbar_n(u_0u_1\ldots u_{n-1},\omega^{(n)})\,\text{d}\mu_\mathcal{P}(u)<\eps.\]


A process $(\XX,\cP)$ is \emph{$M$-trivial} \cite[Definition 9.2]{Katok} if for any $\eps>0$ there exists $N=N(\eps)$ such that for every $n\geq N$ the process $(\XX,\mathcal P)$ is $(n,\eps)$-trivial.

\begin{thm}[\cite{Katok}, Theorem 4, (1)$\Leftrightarrow$(2)]\label{criterion}
An aperiodic measure-preserving system $\XX$ is loosely Kronecker if and only if for every finite partition $\mathcal{P}$ of $X$ the process $(\XX,\mathcal P)$ is $M$-trivial.
\end{thm}

\subsection{Proof of Theorem \ref{thm:looselyK}} 
In the proof of our main theorem we will need the following fact. 

\begin{lem}\label{lem:Egorov} Let $\mu\in\MTe(X)$ and $\eps>0$. If    $\mathcal{P},\mathcal{R}\in\bP^k(X)$ are such that $d^\mu_1(\mathcal{P},\mathcal{R})<\eps/3$, then there is  $N\in\N$ such that for every 
$n\ge N$ we have
\[
\int_X \fbar_n(\cP^n(x),\cR^n(x))\text{ d}\mu(x)<
\eps.
\]
\end{lem}

\begin{proof} Let $\mathcal{P}=\{P_0,P_1,\ldots,P_{k-1}\}$, $\mathcal{R}=\{R_0,R_1,\ldots,R_{k-1}\}$ and $\Delta=(P_0 \div R_0)\cup(P_1 \div R_1)\cup\ldots\cup(P_{k-1} \div R_{k-1})$. Observe that the ergodicity of $\mu$ implies that for $\mu$-a.e. $x\in X$ we have
\[
\dbar_n(\cP^n(x),\cR^n(x))=\frac{1}{n}\sum_{j=0}^{n-1}\chi_\Delta(T^j(x))\stackrel{(n\to\infty)}{\longrightarrow} \mu(\Delta)=
\mu(\{x\in X:\mathcal{P}(x)\neq\mathcal{Q}(x)\})=d^\mu_1(\mathcal{P},\mathcal{Q}).
\]
It follows from the Egorov theorem that there are a measurable set $G\subset X$ and $N\in\N$ such that $\mu(G)>1-\eps/3$ and for all $n\geq N$ and $x\in G$ one has $\dbar_n(\cP^n(x),\cR^n(x))<d^\mu_1(\cP,\cR)+\eps/3$.
Notice that $\fbar_n(\cP^n(x),\cR^n(x))\le \dbar_n(\cP^n(x),\cR^n(x))$ for all $n\in\N$.
Therefore, if $d^\mu_1(\cP,\cR)<\eps/3$ and $n\ge N$, then
\begin{multline*}
\int_X \fbar_n(\cP^n(x),\cR^n(x))\text{ d}\mu(x)\le \int_X \dbar_n(\cP^n(x),\cR^n(x))\text{ d}\mu(x)\le \\\le \int_G  \dbar_n(\cP^n(x),\cR^n(x))\text{ d}\mu(x) +\mu(X\setminus G)\le d^\mu_1(\cP,\cR)+\eps/3+\eps/3  \le\eps. \qedhere
\end{multline*}
\end{proof}

The rest of this section is devoted to the proof of Theorem \ref{thm:looselyK}.
\begin{proof}[Proof of Theorem \ref{thm:looselyK}]
Let $\mu=\mu^{(0)}$ be an $\Fbar$-limit of a sequence of 
loosely Kronecker measures $\mu^{(1)},\mu^{(2)},\ldots$. Then there is a sequence $(x^{(n)})_{n=1}^\infty\subseteq X$ and  a generic quasi-orbit $\und x$ for $\mu^{(0)}$ such that  $\Fbar(\und{x},\und{x}^{(n)}_T)\to 0$ as $n\to\infty$ and $x^{(n)}$ is a generic point for  $\mu^{(n)}$ for $n\in \N$. 
Then $\mu^{(0)}$ is ergodic. If $\mu^{(0)}$ is a periodic measure, then there is nothing to show. Assume that $\mu^{(0)}$ is aperiodic.

To apply Katok's criterion (Theorem \ref{criterion}) we need to show that for every finite partition $\mathcal P$ of $X$ the process $(\XX,\mathcal P)$, where $\XX=(X,X_{\mathscr{B}},\mu^{(0)},T)$ is $M$-trivial. To this end we choose a partition $\mathcal{P}\in\bP^k(X)$ and fix $\eps>0$.

We use Lemma \ref{lem:fbar-bound} to find $0<\alpha<\eps/6$ such that $\fbar(\und{\omega},\und{\omega}')<\alpha$ for some $\und{\omega}\in\Gen(\xi)$, $\und{\omega}'\in\Gen(\xi')$, where $\xi$ and $\xi'$ are shift invariant ergodic measures on $\Omega_{k+1}$ implies that $\fbar(\xi,\xi')<\eps/3$.

We take $\delta<\min\{\alpha/2,\eps/9\}$. 
We apply Lemma \ref{lem:repartition} to $(\mu^{(n)})_{n=0}^\infty$ to find $\gamma=\gamma(\mu=\mu^{(0)},\delta,\mathcal{P})$ and a partition $\mathcal{R}\in\bP^{k+1}(X)$ with $d_1^{\mu^{(0)}}(\cP,\cR)<\delta$ and $\mu^{(n)}(\partial\mathcal R)=0$ for $n=0,1,\ldots$.  Let $N$ 
be such that $\Fbar(\und x,x^{(N)}_T)<\min\{\gamma,\alpha/2\}$. 
It follows from Lemma \ref{lem:repartition} that 
\begin{equation}\label{eq:fbar-close-points}
\fbar(\mathcal{R}(\und{x}),\mathcal{R}(\und{x}^{(N)}_T))<\alpha<\eps/6.
  \end{equation}

By Lemma \ref{lem:generic}, 
the point $\mathcal{R}(\und{x}^{(N)}_T)$ is generic for $\mathcal{R}$-representation measure of $\mu^{(N)}$ denoted $\cR_*\mu^{(N)}=\mu^{(N)}_\mathcal{R}$ and $\mathcal{R}(\und x)$ is generic for the $\mathcal{R}$-representation $\mu^{(0)}_\mathcal{R}=\cR_*\mu^{(0)}$ of $\mu^{(0)}$. The inequality \eqref{eq:fbar-close-points} and our choice of $\alpha$ imply that 
$\fbar(\mu^{(0)}_\mathcal{R},\mu_\mathcal{R}^{(N)})<\eps/3$. Hence, by the definition of $\fbar$ metric, we find $M'=M'(\eps)>0$ satisfying that for every $m\ge M'$ there is a coupling $\lambda_{m}$ of $\cR_*\mu|_{\cR^m}$ and $\cR_*\mu^{(N)}|_{\cR^m}$ for which we have
\begin{equation}\label{ineq:coupling}
\int_{\cR^m\times\cR^m}\fbar_m(\zeta,\zeta')\text{ d}\lambda_{m}(\zeta,\zeta')<\eps/3.    
\end{equation}

Using that $\mu^{(N)}$ is loosely Kronecker or periodic  
we find $M''=M''(\eps)$ such that for every $m\ge M''$ there is $\omega^{(m)}\in\cR^m$ satisfying
\begin{equation}\label{ineq:LK}
\int_X \fbar_m(\cR^m(x),\omega^{(m)})\text{ d}\mu^{(N)}(x)<\eps/3.    
\end{equation}

With the above notation, 
let $m\ge \max\{M',M''\}$. Using \eqref{ineq:coupling} and \eqref{ineq:LK} we get (below, $\cR^{2m}=\cR^m\times\cR^m$)
\begin{align}
\begin{split}\label{ineq:final}
\int_X \fbar_m(\cR^m(x),\omega^{(m)})\text{ d}\mu^{(0)}(x)&=
\int_{\cR^{2m}}
\fbar_m(\zeta,\omega^{(m)})\text{ d}\lambda_m(\zeta,\zeta')\\
&\le \int_{\cR^{2m}}
(\fbar_m(\zeta,\zeta')
+\fbar_m(\zeta',
\omega^{(m)}))\text{ d}\lambda_m(\zeta,\zeta')\\
&=\int_{\cR^{2m}}
\fbar_m(\zeta,\zeta')\text{ d}\lambda_m(\zeta,\zeta')
+\int_{X} \fbar_m(\cR^m(y),
\omega^{(m)})\text{ d}\mu^{(N)}(y)\\
&\le 2\eps/3.    
\end{split}
\end{align}

By Lemma \ref{lem:Egorov}, 
since $\cR$ is a partition satisfying $d^{\mu^{(0)}}_1(\cP,\cR)<\delta<\eps/9$, there exists $M\in\N$ 
such that for every 
$m\ge M$ we have
\begin{equation}\label{ineq:Egorov}
\int_X \fbar_m(\cP^m(x),\cR^m(x))\text{ d}\mu^{(0)}(x)<\eps/3.    
\end{equation}
Combining \eqref{ineq:LK} with \eqref{ineq:final} we see that 
if $m\ge N=\max\{M,M',M''\}$, then
\[
\int_X \fbar_m(\cP^m(x),\omega^{(m)})\text{ d}\mu^{(0)}(x)<\eps.
\]
Since $\eps>0$ was arbitrary, we have proved that $(\XX,\mathcal{P})$ is $M$-trivial. Since $\mathcal{P}$ was also arbitrary and $\mu^{(0)}$ is aperiodic we conclude using Theorem \ref{criterion} that $\mu^{(0)}$ is loosely Kronecker (standard).
\end{proof}

As a consequence of Theorem~\ref{thm:GIKNmain}, Theorem~\ref{thm:Cauchy}, and Theorem~\ref{lsc} (combined with Theorem \ref{thm:BDG} for the second part), we get the following result, which we still call a theorem due to its importance.

\begin{thm}\label{thm:GIKN-improved}
If $\mu$ is an aperiodic royal measure resulting from a GIKN-sequence $(\Gamma_n)_{n\in\N}$, then $\mu$ is a loosely Kronecker measure (hence, ergodic and with zero entropy). 
In particular, under the assumptions of Theorem \ref{thm:BDG} the resulting measure $\mu$ is either a nonhyperbolic periodic measure or a nonhyperbolic loosely Kronecker measure (hence, $\mu$ is ergodic and has zero entropy).
In all cases above, $\mu$ is supported on
\[
\supp\mu=\bigcap_{k=1}^\infty\overline{\bigcup_{n\ge k}\Gamma_n}.
\]
\end{thm}



Theorem \ref{thm:GIKN-improved} 
may replace \cite[Lemma 2.5]{BDG} and it reveals more information about the resulting measure. This applies for example to \cite[Theorem B]{CCGWY} or \cite[Proposition 1.1]{CCGWY}. Similar strengthenings are possible for results from \cite{BC,BBD, BDG, BZ, DG, DGRZ, GIKN, KN}. We leave the details to the reader, since presenting them here would require repeating a lot of material from these papers without introducing anything new.

\bibliographystyle{plain}
\bibliography{bib}

\begin{thebibliography}{10}

\bibitem{BCKO}
Sejal Babel, Melih~Emin Can, Dominik Kwietniak, and Piotr Oprocha.
\newblock Spectrum of invariant measures via generic points.
\newblock preprint, 2025.

\bibitem{BL}
Sejal Babel and Martha \L\k{a}cka.
\newblock On the closedness of ergodic measures in a characteristic class.
\newblock {\em Colloq. Math.}, 178(1):41--47, 2025.

\bibitem{BC}
Pablo~G. Barrientos and Joel~Angel Cisneros.
\newblock Zero {L}yapunov exponents in transitive skew-products of iterated
  function systems, 2024.
\newblock Available at \url{https://arxiv.org/abs/2403.11040}.

\bibitem{BFK}
Fran\c{c}ois Blanchard, Enrico Formenti, and Petr Kůrka.
\newblock Cellular automata in the {C}antor, {B}esicovitch, and {W}eyl
  topological spaces.
\newblock {\em Complex Systems}, 11(2):107--123, 1997.

\bibitem{BBD}
Jairo Bochi, Christian Bonatti, and Lorenzo~J. D\'{\i}az.
\newblock Robust vanishing of all {L}yapunov exponents for iterated function
  systems.
\newblock {\em Math. Z.}, 276(1-2):469--503, 2014.

\bibitem{BBD2}
Jairo Bochi, Christian Bonatti, and Lorenzo~J. D\'{\i}az.
\newblock Robust criterion for the existence of nonhyperbolic ergodic measures.
\newblock {\em Comm. Math. Phys.}, 344(3):751--795, 2016.

\bibitem{BBD3}
Christian Bonatti, Lorenzo~J. D\'{\i}az, and Jairo Bochi.
\newblock A criterion for zero averages and full support of ergodic measures.
\newblock {\em Mosc. Math. J.}, 18(1):15--61, 2018.

\bibitem{BDG}
Christian Bonatti, Lorenzo~J. D\'{\i}az, and Anton Gorodetski.
\newblock Non-hyperbolic ergodic measures with large support.
\newblock {\em Nonlinearity}, 23(3):687--705, 2010.

\bibitem{BDK}
Christian Bonatti, Lorenzo~J. D\'{\i}az, and Dominik Kwietniak.
\newblock Robust existence of nonhyperbolic ergodic measures with positive
  entropy and full support.
\newblock {\em Ann. Sc. Norm. Super. Pisa Cl. Sci. (5)}, 22(4):1643--1672,
  2021.

\bibitem{BZ}
Christian Bonatti and Jinhua Zhang.
\newblock On the existence of non-hyperbolic ergodic measures as the limit of
  periodic measures.
\newblock {\em Ergodic Theory Dynam. Systems}, 39(11):2932--2967, 2019.

\bibitem{BJK}
Damla Buldağ, Bhishan Jacelon, and Dominik Kwietniak.
\newblock Even the vague specification property implies density of ergodic
  measures, 2025.
\newblock Available at \url{https://arxiv.org/abs/2501.17820}.

\bibitem{CaiKwLiPourmand-JDE-2022}
Fangzhou Cai, Dominik Kwietniak, Jian Li, and Habibeh Pourmand.
\newblock On the properties of the mean orbital pseudo-metric.
\newblock {\em J. Differential Equations}, 318:1--19, 2022.

\bibitem{CaiLi-Nonlinearity-2023}
Fangzhou Cai and Jie Li.
\newblock On {F}eldman--{K}atok metric and entropy formulas.
\newblock {\em Nonlinearity}, 36(9):4758--4784, 2023.

\bibitem{CT}
Melih~Emin Can and Alexandre Trilles.
\newblock On the weakness of the vague specification property.
\newblock {\em Nonlinearity}, 38(8):Paper No. 085004, 28, 2025.

\bibitem{CCGWY}
Cheng Cheng, Sylvain Crovisier, Shaobo Gan, Xiaodong Wang, and Dawei Yang.
\newblock Hyperbolicity versus non-hyperbolic ergodic measures inside
  homoclinic classes.
\newblock {\em Ergodic Theory Dynam. Systems}, 39(7):1805--1823, 2019.

\bibitem{Diaz}
Lorenzo~J. D\'{\i}az.
\newblock Nonhyperbolic ergodic measures.
\newblock In {\em Proceedings of the {I}nternational {C}ongress of
  {M}athematicians---{R}io de {J}aneiro 2018. {V}ol. {III}. {I}nvited
  lectures}, pages 1887--1907. World Sci. Publ., Hackensack, NJ, 2018.

\bibitem{DG}
Lorenzo~J. D\'{\i}az and Anton Gorodetski.
\newblock Non-hyperbolic ergodic measures for non-hyperbolic homoclinic
  classes.
\newblock {\em Ergodic Theory Dynam. Systems}, 29(5):1479--1513, 2009.

\bibitem{Downarowicz}
Tomasz Downarowicz.
\newblock {\em Entropy in dynamical systems}, volume~18 of {\em New
  Mathematical Monographs}.
\newblock Cambridge University Press, Cambridge, 2011.

\bibitem{DKL}
Tomasz Downarowicz, Dominik Kwietniak, and Martha Łącka.
\newblock Uniform continuity of entropy rate with respect to the {$\overline
  f$}-pseudometric.
\newblock {\em IEEE Trans. Inform. Theory}, 67(11):7010--7018, 2021.

\bibitem{DGRZ}
Lorenzo~J. Díaz, Katrin Gelfert, Michal Rams, and Jinhua Zhang.
\newblock Full flexibility of entropies among ergodic measures for partially
  hyperbolic diffeomorphisms, 2025.
\newblock Available at \url{https://arxiv.org/abs/2403.11040}.

\bibitem{Feldman}
J.~Feldman.
\newblock New {$K$}-automorphisms and a problem of {K}akutani.
\newblock {\em Israel J. Math.}, 24(1):16--38, 1976.

\bibitem{FN}
J.~Feldman and D.~Nadler.
\newblock Reparametrization of {$n$}-flows of zero entropy.
\newblock {\em Trans. Amer. Math. Soc.}, 256:289--304, 1979.

\bibitem{GaoZhang-ActaMathSin-2024}
Kun~Mei Gao and Rui~Feng Zhang.
\newblock On variational principles of metric mean dimension on subsets in
  {F}eldman-{K}atok metric.
\newblock {\em Acta Math. Sin. (Engl. Ser.)}, 40(10):2519--2536, 2024.

\bibitem{GRK}
Felipe Garc\'{\i}a-Ramos and Dominik Kwietniak.
\newblock On topological models of zero entropy loosely {B}ernoulli systems.
\newblock {\em Trans. Amer. Math. Soc.}, 375(9):6155--6178, 2022.

\bibitem{GK}
Marlies Gerber and Philipp Kunde.
\newblock A smooth zero-entropy diffeomorphism whose product with itself is
  loosely {B}ernoulli.
\newblock {\em J. Anal. Math.}, 141(2):521--583, 2020.

\bibitem{GerKun}
Marlies Gerber and Philipp Kunde.
\newblock Non-classifiability of ergodic flows up to time change.
\newblock {\em Invent. Math.}, 239(2):527--619, 2025.

\bibitem{GTW}
E.~Glasner, J.-P. Thouvenot, and B.~Weiss.
\newblock On some generic classes of ergodic measure preserving
  transformations.
\newblock {\em Trans. Moscow Math. Soc.}, 82:15--36, 2021.

\bibitem{Glasner}
Eli Glasner.
\newblock {\em Ergodic theory via joinings}, volume 101 of {\em Mathematical
  Surveys and Monographs}.
\newblock American Mathematical Society, Providence, RI, 2003.

\bibitem{GIKN}
A.~S. Gorodetski\u{\i}, Yu.~S. Ilyashenko, V.~A. Kleptsyn, and M.~B.
  Nalski\u{\i}.
\newblock Nonremovability of zero {L}yapunov exponents.
\newblock {\em Funktsional. Anal. i Prilozhen.}, 39(1):27--38, 95, 2005.

\bibitem{Gu}
Rongbao Gu.
\newblock The asymptotic average shadowing property and transitivity.
\newblock {\em Nonlinear Anal.}, 67(6):1680--1689, 2007.

\bibitem{Katok}
A.~B. Katok.
\newblock Monotone equivalence in ergodic theory.
\newblock {\em Izv. Akad. Nauk SSSR Ser. Mat.}, 41(1):104--157, 231, 1977.

\bibitem{KatokSataev}
A.~B. Katok and E.~A. Sataev.
\newblock Standardness of rearrangement automorphisms of segments and flows on
  surfaces.
\newblock {\em Mat. Zametki}, 20(4):479--488, 1976.

\bibitem{KN}
V.~A. Kleptsyn and M.~B. Nalski\u{\i}.
\newblock Stability of the existence of nonhyperbolic measures for
  {$C^1$}-diffeomorphisms.
\newblock {\em Funktsional. Anal. i Prilozhen.}, 41(4):30--45, 96, 2007.

\bibitem{KKO}
Marcin Kulczycki, Dominik Kwietniak, and Piotr Oprocha.
\newblock On almost specification and average shadowing properties.
\newblock {\em Fund. Math.}, 224(3):241--278, 2014.

\bibitem{KLO2}
Dominik Kwietniak, Martha \L\k{a}cka, and Piotr Oprocha.
\newblock A panorama of specification-like properties and their consequences.
\newblock In {\em Dynamics and numbers}, volume 669 of {\em Contemp. Math.},
  pages 155--186. Amer. Math. Soc., Providence, RI, 2016.

\bibitem{KLO}
Dominik Kwietniak, Martha \L\k{a}cka, and Piotr Oprocha.
\newblock Generic points for dynamical systems with average shadowing.
\newblock {\em Monatsh. Math.}, 183(4):625--648, 2017.

\bibitem{KL-arXiv-2017}
Dominik Kwietniak and Martha Łącka.
\newblock Feldman-{K}atok pseudometric and the {G}{I}{K}{N} construction of
  nonhyperbolic ergodic measures, 2017.
\newblock Available at \url{https://arxiv.org/abs/1702.01962}.

\bibitem{Lacka}
Martha \L\k{a}cka.
\newblock Non-hyperbolic ergodic measures with the full support and positive
  entropy.
\newblock {\em Monatsh. Math.}, 200(1):163--178, 2023.

\bibitem{Ornstein}
Donald Ornstein.
\newblock Factors of {B}ernoulli shifts are {B}ernoulli shifts.
\newblock {\em Advances in Math.}, 5:349--364 (1970), 1970.

\bibitem{ORW}
Donald~S. Ornstein, Daniel~J. Rudolph, and Benjamin Weiss.
\newblock Equivalence of measure preserving transformations.
\newblock {\em Mem. Amer. Math. Soc.}, 37(262):xii+116, 1982.

\bibitem{Oxtoby}
John~C. Oxtoby.
\newblock Ergodic sets.
\newblock {\em Bull. Amer. Math. Soc.}, 58:116--136, 1952.

\bibitem{Sataev}
E.~A. Sataev.
\newblock An invariant of monotone equivalence that determines the factors of
  automorphisms that are monotonely equivalent to the {B}ernoulli shift.
\newblock {\em Izv. Akad. Nauk SSSR Ser. Mat.}, 41(1):158--181, 231--232, 1977.

\bibitem{Shields}
Paul~C. Shields.
\newblock {\em The ergodic theory of discrete sample paths}, volume~13 of {\em
  Graduate Studies in Mathematics}.
\newblock American Mathematical Society, Providence, RI, 1996.

\bibitem{Sigmund}
Karl Sigmund.
\newblock On minimal centers of attraction and generic points.
\newblock {\em J. Reine Angew. Math.}, 295:72--79, 1977.

\bibitem{Trilles}
Alexandre Trilles.
\newblock A characterization of zero entropy loosely {B}ernoulli flows via
  {FK}-pseudometric.
\newblock {\em Proc. Amer. Math. Soc.}, 153(2):755--771, 2025.

\bibitem{Tru}
Frank Trujillo.
\newblock Smooth, mixing transformations with loosely {B}ernoulli {C}artesian
  square.
\newblock {\em Ergodic Theory Dynam. Systems}, 42(3):1252--1283, 2022.

\bibitem{Weiss}
Benjamin Weiss.
\newblock {\em Single orbit dynamics}, volume~95 of {\em CBMS Regional
  Conference Series in Mathematics}.
\newblock American Mathematical Society, Providence, RI, 2000.

\bibitem{XieChenYang-JDCS-2024}
Yunxiang Xie, Ercai Chen, and Jiao Yang.
\newblock Weighted entropy formulae on {F}eldman-{K}atok metric.
\newblock {\em J. Dyn. Control Syst.}, 30(2):Paper No. 21, 16, 2024.

\bibitem{XieChenYang-DynamicalSystems-2025}
Yunxiang Xie, Ercai Chen, and Rui Yang.
\newblock Variational principles for {F}eldman-{K}atok metric mean dimension.
\newblock {\em Dyn. Syst.}, 40(1):23--34, 2025.

\end{thebibliography}

\end{document}